\documentclass[graybox]{svmult}

\usepackage{ dsfont }
\usepackage{enumitem}
\usepackage{mathptmx}       
\usepackage{helvet}         
\usepackage{courier}        
\usepackage{type1cm}        
\usepackage{makeidx}         
\usepackage{amsfonts}
\usepackage{amsmath}
\usepackage{amssymb}
\usepackage{amsmath}
\usepackage[english]{babel}
\usepackage{hyperref}
\input xy
\usepackage[all]{xy}

\usepackage{tikz}

\hypersetup{colorlinks,citecolor=blue}

\input xy
\xyoption{all}

\newcommand{\PSL}{\mathrm{PSL}}
\newcommand{\SL}{\mathrm{SL}}
\newcommand{\GL}{\mathrm{GL}}
\newcommand{\PGL}{\mathrm{PGL}}
\newcommand{\C}{\mathbb{C}}
\newcommand{\Pp}{\mathbb{P}}
\newcommand{\Z}{\mathbb{Z}}
\newcommand{\J}{\mathbb{J}}
\newcommand{\R}{\mathbb{R}}
\newcommand{\FL}{\mathcal{F}(\PCC)}

\newcommand{\PC}{\mathbb{P}(\mathbb{C}^2)}
\newcommand{\PCC}{\mathbb{P}(\mathbb{C}^3)}
\newcommand{\Vol}{\mathrm{Vol}}
\newcommand{\h}{\mathrm{H}}
\newcommand{\hb}{\mathrm{H}_\mathrm{b}}
\newcommand{\hc}{\mathrm{H}_\mathrm{c}}
\newcommand{\hcb}{\mathrm{H}_\mathrm{cb}}
\newcommand{\QFJ}{Q(\mathbb{F},\mathds{J})}
\newcommand{\fFJ}{f(\mathbb{F},\mathds{J})}

\makeindex             

\newtheorem{thm}[theorem]{Theorem}
  \newtheorem{lem}[theorem]{Lemma}
  \newtheorem{prop}[theorem]{Proposition}
 \newtheorem{cor}[theorem]{Corollary}

\begin{document}

\title*{The bounded Borel class and complex representations of $3$-manifold groups}
\author{Michelle Bucher, Marc Burger and Alessandra Iozzi}
\authorrunning{M. Bucher, M. Burger and A. Iozzi} 
\institute{Michelle Bucher \at Section de Math\'ematiques Universit\'e de Gen\`eve, 
2-4 rue du Li\`evre, Case postale 64, 1211 Gen\`eve 4, Suisse, \email{Michelle.Bucher-Karlsson@unige.ch}
\and Marc Burger  \at Department Mathematik, ETH, R\"amistrasse 101, 
8092 Z\"urich, Schweiz, \email{burger@math.ethz.ch}
\and Alessandra Iozzi \at Department Mathematik, ETH, R\"amistrasse 101, 
8092 Z\"urich, Schweiz, \email{iozzi@math.ethz.ch}}

%
%
\maketitle

\abstract{If $\Gamma<\PSL(2,\mathbb{C})$ is a lattice, 
we define an invariant of a representation $\Gamma\rightarrow \PSL(n,\mathbb{C})$ 
using the Borel class $\beta(n)\in \hc^3(\PSL(n,\mathbb{C}),\R)$. 
We show that this invariant satisfies a Milnor--Wod type inequality and 
its maximal value is attained precisely by the representations conjugate
to the restriction to $\Gamma$ of the irreducible complex $n$-dimensional
representation of $\PSL(2,\mathbb{C})$ or its complex conjugate. \\
Major ingredients of independent interest are the extension to degenerate configuration of flags
of a cocycle defined by Goncharov and its study, as well as the identification of 
$\hb^3(\SL(n,\C),\R)$ as a normed space.}


\section{Introduction}\label{sec:intro}
Let $M=\Gamma\backslash\mathbb{H}^3$ be a finite volume quotient of the real hyperbolic $3$-space $\mathbb{H}^3$,
where $\Gamma<\mathrm{PSL}(2,\mathbb{C})$ is a torsion free lattice.  
There is a considerable body of work concerning the representation variety $\mathrm{Hom}(\Gamma,\PSL(n,\mathbb{C}))$, 
the problem of finding explicit parametrizations of (large parts of) it, 
and expressing various invariants of such representations like the ``volume'', 
the Dehn invariant and the Chern--Simons invariant in those parameters, 
\cite{Garoufalidis_Thurston_Zickert, Bergeron_Falbel_Guilloux, Dimofte_Gabella_Goncharov}. 
In fact this representation variety is particularly rich when $M$ is not compact, 
say with $h\geq1$ cusps, since for instance in this case the character variety of $\Gamma$ into $\PSL(n,\mathbb{C})$ 
locally has dimension $(n-1)h$  near $\pi_n\mid_\Gamma$, 
where $\pi_n:\PSL(2,\mathbb{C})\to \PSL(n,\mathbb{C})$ is the irreducible complex representation, \cite{Menal_Porti}. 

In this paper, we will study the volume of a representation $\rho:\Gamma\to \PSL(n,\mathbb{C})$ 
that we will rename as the {\em Borel invariant of $\rho$}. 
Indeed, the continuous cohomology of $\PSL(n,\mathbb{C})$ in degree $3$ is generated 
by a specific class called the {\em Borel class} $\beta(n)$. 
When $M$ is compact, the definition of the Borel invariant of $\rho$ is straightforward 
as it is the evaluation on the fundamental class $[M]$ of the pullback by $\rho$ of the Borel class. 
If $M$ has cusps, the definition of this invariant presents interesting difficulties 
which we overcome by the use of bounded cohomology. 
More precisely, $\beta(n)$ can be represented by a bounded cocycle,
which gives rise to a bounded continuous class
\begin{equation*}
\beta_\mathrm{b}(n)\in\hcb^3(\PSL(n,\mathbb{C}),\R).
\end{equation*}
The Borel invariant of $\rho:\Gamma\to \PSL(n,\mathbb{C})$ is then defined as 
\begin{equation*}
\mathcal{B}(\rho)=\langle \rho^*(\beta_\mathrm{b}(n)),[N,\partial N]\rangle,
\end{equation*}
where $N$ is a compact core of $M$. 
We refer the reader to Section~\ref{sec:invariant} for a precise interpretation of this formula. This definition does not use any triangulation, it is independent of the choice of compact core and can be made for any compact oriented $3$-manifold whose boundary has amenable fundamental group. 

The bounded cocycle entering the definition of $\beta_\mathrm{b}(n)$ is constructed by means of an invariant
$$B_n:\mathcal{F}(\mathbb{C}^n)^4\longrightarrow \R$$
of $4$-tuples of complete flags, which on generic $4$-tuples has been defined and studied by A.B. Goncharov, \cite{Goncharov}. 
It generalizes the volume function in the case $\mathcal{F}(\mathbb{C}^2)=  \PC=\partial \mathbb{H}^3$ (see  Section~\ref{sec:invariant} for a detailed discussion). 
This invariant can also be used to give an efficient formula for $\mathcal{B}(\rho)$. 
To this end assume that $M$ has toric cusps.  Let $\varphi:\mathcal{C}\to\mathcal{F}(\mathbb{C}^n)$ be a decoration, that is any $\Gamma$-equivariant
map from the set of cusps $\mathcal{C}\subset\partial\mathbb{H}^3$ into $\mathcal{F}(\mathbb{C}^n)$, and let $P_1,\dots, P_r$
be a family of oriented ideal tetrahedra with vertices in $\mathcal{C}$ forming an ideal triangulation of $M$.
If $(P_i^0,P_i^1,P_i^2,P_i^3)$ are the vertices of $P_i$, then
\begin{equation}\label{eq:interpr}
\mathcal{B}(\rho)=\sum_{i=1}^rB_n(\varphi(P_i^0),\varphi(P_i^1),\varphi(P_i^2),\varphi(P_i^3))
\end{equation}
(see Section~\ref{sec:invariant} for a proof). 
Notice that our formula for the volume does not involve any barycentric subdivision of the ideal triangulation, 
nor any conditions on the decoration. 
Upon passing to a barycentric subdivision, or restricting to generic decorations,  
one recovers from the right hand side of \eqref{eq:interpr} the formulas in 
\cite{Garoufalidis_Thurston_Zickert, Bergeron_Falbel_Guilloux, Dimofte_Gabella_Goncharov} .

Our first result is that on the character variety $\Gamma$ into $\PSL(n,\mathbb{C})$, 
the invariant $\mathcal{B}$ attains a unique maximum at $[\pi_n|_\Gamma]$.

\begin{theorem}  \label{thm: main} \label{theorem:maxrep} 
Let $\Gamma=\pi_1(M)$ be the fundamental group of a complete finite volume real hyperbolic $3$-manifold and
let $\rho:\Gamma\to\mathrm{PSL}(n,\mathbb{C})$ be any representation.  Then
\begin{equation*} 
|\mathcal{B}(\rho)|\leq\frac{n(n^2-1)}{6}\Vol(M)\,,
\end{equation*}
with equality if and only if $\rho$ is conjugate to $\pi_n|_\Gamma$ or to its complex conjugate $\overline\pi_n|_\Gamma$.
\end{theorem}

The case of the character variety of $\Gamma$ into $\PSL(3,\mathbb{C})$ is instructive: 
in \cite{BFGKR} the authors study the derivative of $\mathcal{B}$ on a Zariski open subset and show 
that it is entirely expressed in terms of the eigenvalues of the holonomy at the cusps.
In particular boundary unipotent representations are critical points of $\mathcal{B}(\rho)$.
The example of the complement of the figure eight knot \cite{Bergeron_Falbel_Guilloux}
suggests that in general there are many boundary
unipotent representations and therefore many critical points for $\mathcal{B}$.  

A large part of this paper is devoted to the study of the invariant $B_n:\mathcal{F}(\mathbb{C}^n)^4\to\R$  on $4$-tuples of flags
(see Theorem~\ref{thm: key thm} below), to the bounded class it defines and 
the consequences, in combination with stability results by N. Monod in \cite{Monod},
for the bounded cohomology of $\PSL(n,\mathbb{C})$. 
Our main result about the bounded cohomology of these groups in degree $3$ is: 

\begin{theorem}\label{thm: 2 intro} \label{thm 2 intro} The class $\beta_\mathrm{b}(n)$ is a generator of $\hcb^3(\mathrm{PSL}(n,\mathbb{C}),\R)$ 
and its Gromov norm is
\begin{equation*} 
\|\beta_\mathrm{b}(n)\|=\frac{n(n^2-1)}{6}v_3\,,
\end{equation*}
where $v_3$ is the volume of a maximal ideal tetrahedron in $\mathbb{H}^3$.  
In addition $\beta_\mathrm{b}(n)$ restricts to $\beta_\mathrm{b}(n-1)$ under the left corner
injection $\mathrm{SL}(n-1,\mathbb{C})\hookrightarrow\mathrm{SL}(n,\mathbb{C})$ and to $(n(n^2-1)/6)\cdot \beta_\mathrm{b}(2)$ 
under the irreducible representation $\pi_n:\mathrm{SL}(2,\mathbb{C})\rightarrow\mathrm{SL}(n,\mathbb{C})$.
\end{theorem}

The case $n=2$ follows from work of Bloch \cite{Bloch}. 
Theorem~\ref{thm: 2 intro} gives additional evidence for the conjecture that for simple connected Lie groups with finite center, 
the comparison map between bounded continuous and continuous cohomology is an isomorphism. 
So far this conjecture has been established only in degree $2$ \cite{Burger_Monod_JEMS}, 
in degree $3$ for the isometry group of real hyperbolic $n$-space \cite{Pieters}, 
and in degree $3$ and $4$ for $\SL(2,\R)$ (\cite{Burger_Monod_ern} and \cite{Hartnick_Ott} respectively).  

\section{Outline of the Paper and the definition of the Borel invariant} \label{sec:invariant}

\subsection*{The cocycle representing $\beta_\mathrm{b}(n)$}

We start in Sections~\ref{new section def cocycle} and \ref{sec:affine_flags} 
by setting up a homological machinery involving chains on configuration spaces; 
this is largely borrowed from Goncharov, \cite{Goncharov}. 
The aim is to define an invariant
\begin{equation*}
B_n:\mathcal{F}(\mathbb{C}^n)^4\longrightarrow \R
\end{equation*}
on the space of $4$-tuples of complete flags in $\mathbb{C}^n$ and to show that it is a strict cocycle. 
The definition of the cocycle in general is rather technical, so we will illustrate here 
only the case $n=3$ in the dual setting, that is interchanging lines and planes through the origin. 
Observe that interestingly, Falbel and Wang have just proven \cite{Falbel_Wang} 
that this dual cocycle is not equal to GoncharovÕs cocycle but only cohomologous to it. 
Furthermore the authors give an explicit coboundary for the difference of the two cocycles \cite[Proposition 3.10]{Falbel_Wang}.

A complete flag in $\mathbb{C}^2$ is a choice of a line in $\mathbb{C}^2$ or, equivalenly, of a point $P\in\PC$.
Using the identification $\Pp^2(\mathbb{C})=\partial \mathbb{H}^3$,  the invariant $B_2$ associates to four points
in $\PC$ the signed volume of the ideal tetrahedron that they define.

After projectivization, a complete flag $F$ in $\PCC$ is given by a projective line $L\subset \PCC$ and 
a point $P\in L$. We denote it by $F=\{P\in L\subset \PCC \}\in\FL$. 
Given a complete flag $F\in \FL$ and a projective line $L'\subset \PCC$, 
we define the intersection $F\cap L'$ to be the point in $\PCC$ given by
\begin{equation*}
F\cap L'
=\begin{cases}
L\cap L' &\mathrm{if }\ L\neq L',\\
P &\mathrm{if }\  L=L'.
\end{cases}
\end{equation*}

Now we define a cochain $\beta_3:\FL^4\rightarrow \R$ by sending four flags 
$F_0,\dots,F_3$, where $F_i=\{ P_i\in L_i\subset \PCC\}$, to
\begin{equation*}
\beta_3(F_0,\dots,F_3)=
\begin{cases}
\Vol_{L_i}(F_0\cap L_i,\dots,F_3\cap L_i)&\mathrm{if} \ \exists\, i\neq j \mathrm{\ with \ } L_i=L_j,\\
\Vol_L(F_0\cap L,\dots,F_3\cap L) &\mathrm{if \ }\cap_{j=0}^3 L_j \mathrm{\ is \ a \ point,}\\
\sum_{i=0}^3  \Vol_{L_i}(F_0\cap L_i,\dots,F_3\cap L_i)\quad & \mathrm{otherwise,}
\end{cases}
\end{equation*}
where in the second case, 
$L$ is any projective line not containing the point $\cap_{j=0}^3 L_j$
and $\Vol_L=B_2$ (respectively $\Vol_{L_i}=B_2$) after the identification 
of $L_i$ (respectively of $L$) with $\PC$. To check that $\beta_3$ is well defined we need some observations.

\begin{figure}
\begin{tikzpicture}
\draw ( -6.5,-6/38) -- (3,1.5);
\draw (3,2) node {$L_0$};
\filldraw (0,705/722) circle [radius=0.05];
\draw (0,705/722+.6) node {$P_0$};
\draw (1,-2.5) -- (1,3);
\draw (1,3.5) node {$L_1$};
\filldraw (1,2) circle [radius=0.05];
\draw (1.3,2) node {$P_1$};
\draw ( -6,1) -- (2,-2);
\draw (-6.5,1) node {$L_2$};
\filldraw (0,-5/4) circle [radius=0.05];
\draw (0,-5/4+.3) node {$P_2$};
\draw (-7,0) -- (4,0);
\draw (-7.5,0) node {$L_3$};
\filldraw (-5,0) circle [radius=0.05];
\draw (-5,-.3) node {$P_3$};
%
\draw (114/126-13/2,0) circle [radius=0.2];
\draw (-4.1,855/3174) circle [radius=0.2];
\draw (1,831/722) circle [radius=0.2];
\draw (0,705/722) circle [radius=0.2];
\end{tikzpicture}
\endpgfgraphicnamed
\caption{The generic case.}
\label{fig1}
\end{figure}
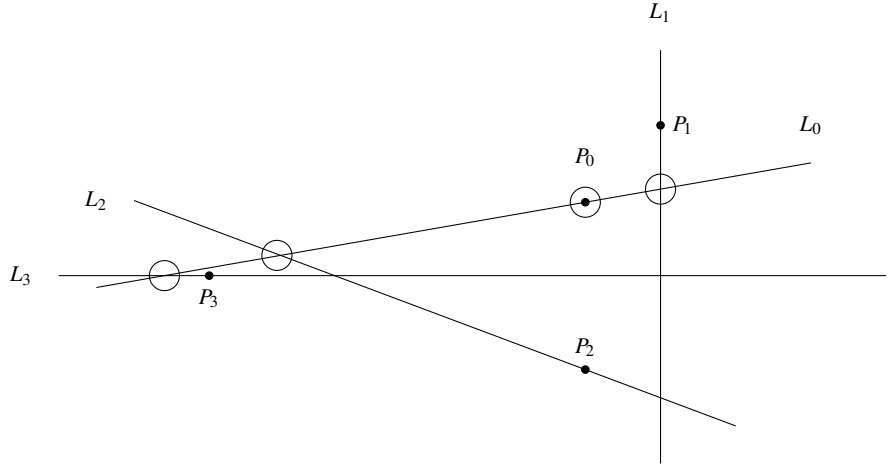

\begin{figure}
\begin{tikzpicture}
\draw ( -6.5,-6/38) -- (3,1.5);
\draw (3,2) node {$L_0=L_2$};
\filldraw (0,705/722) circle [radius=0.05];
\draw (0,705/722+.6) node {$P_0$};
\draw (1,-1) -- (1,3);
\draw (1,3.5) node {$L_1$};
\filldraw (1,2) circle [radius=0.05];
\draw (1.3,2) node {$P_1$};
\filldraw (-2,453/722) circle [radius=0.05];
\draw (-2,453/722+.6) node {$P_2$};
%
\draw (-7,0) -- (4,0);
\draw (-7.5,0) node {$L_3$};
\filldraw (-5,0) circle [radius=0.05];
\draw (-5,-.3) node {$P_3$};
%
\draw (114/126-13/2,0) circle [radius=0.2];
\draw (-2,453/722) circle [radius=0.2];
\draw (1,831/722) circle [radius=0.2];
\draw (0,705/722) circle [radius=0.2];
\end{tikzpicture}
\endpgfgraphicnamed
\caption{The case in which two lines, in this case $L_0$ and $L_2$, coincide (but $P_0\neq P_2$ otherwise the two flags $(P_0,L_0)$ and $(P_2,L_2)$ would be equal, in which case $\Vol=0$.}
\label{fig2}
\end{figure}
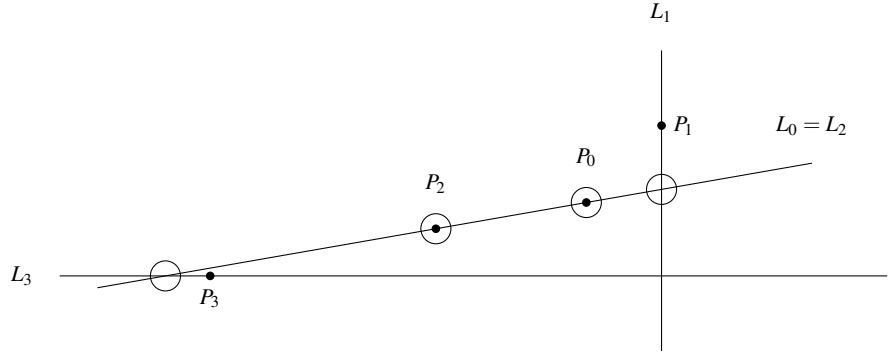

\begin{figure}
\begin{tikzpicture}
\draw ( -6.5,-6/38) -- (3,1.5);
\draw (3,1.8) node {$L$};
%
\draw (1,-1) -- (1,2);
\draw (1,2.3) node {$L_0$};
\filldraw (1,1.5) circle [radius=0.05];
\draw (1.3,1.5) node {$P_0$};
%
%
\draw (-7,0) -- (4,0);
\draw (-7.5,0) node {$L_3$};
\filldraw (-5,0) circle [radius=0.05];
\draw (-5,-.3) node {$P_3$};
\draw (-2,3/2) -- (3,-1);
\draw (-2,3/2+.3) node {$L_2$};
\draw (-344/487,831/974) circle [radius=0.2];
\draw (2,-1/2+.3) node {$P_2$};
\filldraw (2,-1/2) circle [radius=0.05];
\draw (0,2) -- (3/2,-1);
\draw (0,2.3) node {$L_1$};
\draw (701/1494,1586/1494) circle [radius=0.2];
\draw (2/3-.3,2/3) node {$P_1$};
\filldraw (2/3,2/3) circle [radius=0.05];
\draw (114/126-13/2,0) circle [radius=0.2];
\draw (1,831/722) circle [radius=0.2];
\end{tikzpicture}
\endpgfgraphicnamed
\caption{The case in which $\cap_{j=0}^3{L_j}$ is a point.}
\label{fig3}
\end{figure}
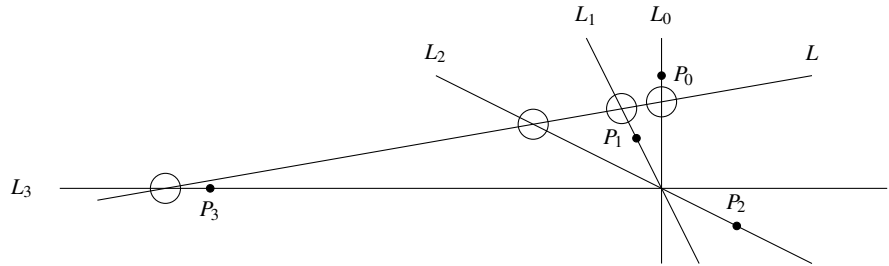

Let $P\in \PCC$ be a point and $L\subset \PCC$ be a projective line not containing $P$. 
We define a projection $p: \PCC \smallsetminus \{P\} \rightarrow L$ by sending a point $x\in \PCC \smallsetminus \{P\}$ 
to the intersection of the line through $P$ and $x$ and the line $L$. 
Note that $p$ is induced by the projection 
$\mathbb{C}^3=\overline{P}\oplus \overline{L}\rightarrow \overline{L}$, 
where $\overline{P}$ and $\overline{L}$ are the vector spaces corresponding to $P$ and $L$ respectively. 
The following lemma is immediate, using the fact that $p$ induces an isomorphism between $L'$ and $L$.

\begin{lem} \label{lem: Vol projection} Let $L\subset \PCC$ be a projective line and $P\in \PCC \smallsetminus L$ a point. 
If $p$ is the projection $p: \PCC \smallsetminus \{P\} \rightarrow L$ defined above, then
$$\Vol_{L'}(x_0,\dots,x_3)=\Vol_L(p(x_0),\dots,p(x_3)),$$
for any projective line $L'\subset \PCC$ not containing $P$ and any $x_0,\dots,x_3\in L'$. 
\end{lem}

Now we can verify that in the second case, 
the definition is independent of the choice of $L$. Indeed, 
let $L'$ be another projective line not containing the point $P=L_0\cap L_1 \cap L_2 \cap L_3$. 
Then the projection $p:\PCC\smallsetminus \{P\}\rightarrow L$ 
sends $F_i\cap L'=L_i\cap L'$ to $F_i\cap L=L_i\cap L$ 
(because $L_i$ is the line containing $P$ and $F_i\cap L'$) and the conclusion follows from Lemma~\ref{lem: Vol projection}. 
Second, observe that the projective line $L_i=L_j$ of the first case might not be uniquely determined. 
This happens precisely if there are two pairs of lines among the four flags which are equal. 
Since $\Vol$ is alternating, we can without loss of generality suppose that $L_0=L_1\neq L_2=L_3$. 
But in this case we have $F_0\cap L_2=F_1\cap L_2=L_0\cap L_2$ and $F_2\cap L_0=F_3 \cap L_0=L_2\cap L_0$, 
so that for any choice of $i\in \{0,1,2,3\}$, 
two of the four points on which the alternating cocycle $\Vol_{L_i}$ will be evaluated are equal, 
so the evaluation is $0$. Finally, it is possible that the first and second case happen simultaneously, 
in which case one easily checks that in both cases one obtains $0$. 

We refer the reader to \eqref{eq:cocycle} in Section~\ref{new section def cocycle} for the definition of $B_n$ for all $n\geq2$
and to \cite{Falbel_Wang}, where $B_3-\beta_3$ is written as an explicit coboundary.

\begin{thm}\label{thm: key thm} \begin{enumerate}
\item The function $B_n$ is a $\mathrm{GL}(n,\mathbb{C})$-invariant alternating strict Borel cocycle.
\item Its absolute value satisfies the inequality
\begin{equation}\label{eq:bound for B_n}
|B_n(F_0,\dots,F_3)|\leq \frac{1}{6}n(n^2-1)v_3
\end{equation} 
with equality 
if and only if $F_0,\dots,F_3$ are, up to the action of $\mathrm{GL}(n,\mathbb{C})$, 
images under the Veronese embedding of vertices of a regular ideal simplex in $\PC=\partial\mathbb H^3$.
\end{enumerate}
\end{thm}

Before outlining the proof of this theorem, we remark that, by evaluation on 
$(g_0F,g_1F,g_2F,g_3F)$, where $F$ is a fixed flag in $\mathcal{F}(\mathbb{C}^n)$, 
the cocycle $B_n$ defines a continuous bounded cohomology class of $\PSL(n,\C)$, 
which we denote by $\beta_\mathrm{b}(n)\in\hcb^3(\PSL(n,\mathbb{C}),\R)$ and call the {\em bounded Borel class}. 
The fact that this class is bounded already follows from Goncharov's almost everywhere defined cocycle,
for which the bound in \eqref{eq:bound for B_n} holds almost everywhere.
However, the stability properties in Theorem~\ref{thm: 2 intro}  under the left corner injection 
as well as the exact determination of the norm require the use of the strict cocycle. 
The proof of Theorem~\ref{thm: 2 intro}  is presented in Section~\ref{sec:hcb3GLnC}. 

The strategy of the proof of Theorem~{thm: key thm} (1) is similar to the one of the Key Lemma of Goncharov \cite[Key Lemma~2.1]{Goncharov}. 
The main modification consists in the fact that we work with spaces of configurations of vectors allowed to be nongeneric.
To show that the function $B_n$ is a strict Borel cocycle we will show that it can be realized
as the pullback via a map of complexes of a cocycle on an appropriate space of ``decorated'' vector spaces.
More precisely, we first introduce the space $\sigma_k$ of isomorphism classes of objects $[V;x_0,\dots,x_k]$ 
consisting of a vector space $V$ and a $(k+1)$-tuple of vectors spanning it and proceed to construct a complex $(\Z[\sigma_k],D_k)$. 
Then we define $\mathrm{Vol}:\sigma_3\rightarrow \R$ using the hyperbolic volume as in Section~\ref{new section def cocycle}
(and hinted at above).
We proceed to show that $D^*_4 \mathrm{Vol}=0$ in Theorem~\ref{thm: Df=0}. 
If $\mathcal{F}_\mathrm{aff}(\mathbb{C}^n)$ is the space of {\em affine flags}  (see Section~\ref{sec:affine_flags}), we finally construct a map of complexes
\begin{equation*}
T_k^*:\R_\mathrm{alt}(\sigma_k)\rightarrow \R_\mathrm{alt}(\mathcal{F}_\mathrm{aff}(\mathbb{C}^n)^{k+1})^{\mathrm{GL}(n,\mathbb{C})}
\end{equation*}
which allows us to view $B_n$ as the pullback $B_n=T_3^*(\mathrm{Vol})$ of the cocycle $\mathrm{Vol}$ on $\sigma_3$ 
to the space of flags and conclude the proof of (1) of Theorem~\ref{thm: key thm}.
In Section~\ref{sec:boundedness} we show the upper bound of $B_n$ by induction in Theorem~\ref{thm: Volm bounded}.
In Section~\ref{sec: max} we analyze the equality case in (2) of Theorem~\ref{thm: key thm}. 
The proof relies in particular on the noteworthy relationship 
\begin{equation*}
B_n(\varphi_n(\xi_0),\ldots,\varphi_n(\xi_{3}))=\frac{1}{6}n(n^2-1)\cdot\Vol_{\mathbb{H}^3}(\xi_0,\dots,\xi_3),
\end{equation*}
for all $\xi_{0},...,\xi_{3}\in \PC$ (see Proposition~\ref{prop: value on irr}), where   $\varphi_n:  \mathbb{P}(\mathbb{C}^2)\to \mathcal{F}(\mathbb{C}^n)$
is the Veronese embedding.
In addition, configurations of maximal $4$-tuples of flags have the same property 
than configurations of $4$-points in $\partial \mathbb{H}^3$ of maximal volume: namely, 
if $B_n(F_0,\dots,F_3)$ is maximal, then $F_3$ is completely determined by $F_0,F_1,F_2$ and the sign of $B_n(F_0,\dots,F_3)$.

\subsection*{The Borel invariant}

Let $\Gamma<\mathrm{PSL}(2,\mathbb{C})$ be a lattice and $\rho:\Gamma\to\PSL(n,\mathbb{C})$ a representation.
We consider first the case in which $\Gamma<\PSL(2,\mathbb{C})$ is torsion-free, so that 
the quotient $M=\Gamma\backslash\mathbb{H}^3$ is a hyperbolic three-manifold
and its cohomology is canonically isomorphic to the cohomology of $\Gamma$. 

If $M$ is compact, then the top dimensional cohomology groups $\h^3(\Gamma,\R)\cong \h^3(M,\R)$ 
are canonically isomorphic to $\mathbb{R}$, where the isomorphism is given by evaluation on the fundamental class $[M]$\footnote{Since we will only consider cohomology with real coefficients, from \S~\ref{new section def cocycle} on for ease of notation the coefficients will be not explicitly stated.}. We define 
$$
\mathcal{B}(\rho)=\langle \rho^*(\beta(n)),[M]\rangle\,,
$$
where $\rho^*:\hc^3(\PSL(n,\mathbb{C}),\mathbb{R} )\rightarrow \h^3(\Gamma,\mathbb{R})$ 
denotes the pull-back via $\rho$. 

If $M$ is not compact, the above definition fails since $\h^3(\Gamma,\mathbb{R})\cong \h^3(M,\mathbb{R})=0$. 
To circumvent this problem we use bounded cohomology, following the approach initiated in 
\cite{Burger_Iozzi_Wienhard_toledo} and used also in \cite{BBI-Mostow}. 
Analogously to what happens in the ordinary group cohomology, a representation $\rho:\Gamma\rightarrow \PSL(n,\mathbb{C})$ 
induces a pullback in bounded cohomology $\rho^*:\hcb^3(\PSL(n,\mathbb{C}),\mathbb{R})\rightarrow \hb^3(\Gamma,\mathbb{R})$ 
and the latter group is canonically isometrically isomorphic to the bounded singular cohomology $\hb^3(M,\mathbb{R})$ of the manifold $M$. 
(The latter fact is true for any CW complex  \cite{Gromov_82, Brooks}, 
but in our case it is a simple consequence of the fact that $M$ is aspherical.) 
Let $N\subset M$ be a compact core of $M$, that is the complement in $M$ of a disjoint union
of finitely many horospherical neighborhoods $E_i$, $i=1,\dots, k$, of cusps.
These have amenable fundamental groups and thus the map
$(N,\partial N)\to(M,\varnothing)$ induces an isometric isomorphism in cohomology, 
$\hb^3(M,\mathbb{R})\cong\hb^3(N,\partial N,\mathbb{R})$, \cite{BBFIPP}, by means of which 
we can consider $\rho^\ast(\beta_\mathrm{b}(n))$ as a bounded relative class in $\hb^3(N,\partial N,\mathbb{R})$.
Finally, the image of $\rho^\ast(\beta_{b}(n))$ via the comparison map 
$c:\hb^n(N,\partial N,\mathbb{R})\to \h^n(N,\partial N,\mathbb{R})$
is an ordinary relative class whose
evaluation on the relative fundamental class $[N,\partial N]$ gives the definition of the Borel invariant of the representation $\rho$,
\begin{equation}\label{eq:invariant}
\mathcal{B}(\rho):=\langle (c\circ\rho^\ast)(\beta_\mathrm{b}(n),[N,\partial N]\rangle\,,
\end{equation}
which is independent of the choice of the compact core $N$.
If $M$ is compact, we recover the invariant previously defined.
We complete the definition in the case in which $\Gamma$ has torsion by setting 
%
\begin{equation*}
\mathcal{B}(\rho):=\frac{\mathcal{B}(\rho|_\Lambda)}{ [\Gamma:\Lambda]}\,,
\end{equation*}
where $\Lambda<\Gamma$ is a torsion free subgroup of finite index.

In order to give a geometric interpretation of this definition when $\Gamma$ is torsion free, we need to understand the composition of the maps 
$$\hcb^3(\PSL(n,\mathbb{C}))\longrightarrow \hb^3(\Gamma)\cong \hb^3(M)\cong \hb^3(N,\partial N)$$
at the cocycle level. 
The difficulty here lies in the isomorphism $\hb^3(\Gamma)\cong \hb^3(N,\partial N)$ and 
we recall from \cite[Section~3]{BBFIPP} that it admits the following explicit description: 
we identify the universal cover $ \widetilde{N}$ of $N$ 
with $\mathbb{H}^3$ minus a $\Gamma$-invariant collection of open horoballs centered at the cusps $\mathcal{C}$. 
Let $p: \widetilde{N}\rightarrow \mathcal{C}$ be the $\Gamma$-equivariant map 
sending each horosphere to the corresponding cusp, and for the interior of $\widetilde{N}$, 
fix a fundamental domain for the $\Gamma$-action, 
map this fundamental domain to a chosen cusp and extend $\Gamma$-equivariantly. 
The bounded cohomology of $\Gamma$ can be represented by $\Gamma$-invariant bounded cocycles 
on the set of cusps of $\Gamma$ in $\partial \mathbb{H}^3$, 
and given such a cocycle $c:\mathcal{C}^4\rightarrow \R$, 
we obtain a relative cocycle on $(N,\partial N)$ 
which we canonically describe as a $\Gamma$-equivariant cocycle on $( \widetilde{N},\partial  \widetilde{N})$ 
as follows:
\begin{equation*}
\{\sigma:\Delta^3\rightarrow \widetilde N\}\mapsto c(p(\sigma(e_0)),\dots, p(\sigma(e_3))),
\end{equation*}
where $e_0,\dots,e_3$ denote the vertices of the standard simplex $\Delta^3$. 

Given a representation $\rho:\Gamma\rightarrow \PSL(n,\mathbb{C})$ 
with $\rho$-equivariant decoration $\varphi:  \mathcal{C}\to\mathcal{F}(\mathbb{C}^n)$, 
it follows from \cite{Burger_Iozzi_app}, 
using the crucial fact that the cocycle $B_n$ is a Borel cocycle defined everywhere,  
that the class $\rho^*(\beta_\mathrm{b}(n))\in \h^3(\Gamma)$ is represented by the cocycle $\varphi^*(B_n)$. 
Thus, given any relative triangulation of $(N,\partial N)$, the Borel invariant of the representation is computable as
\begin{equation}
\mathcal{B}(\rho)=\sum_{i=1}^r B_n(\varphi(P_i^1),\varphi(P_i^2),\varphi(P_i^3),\varphi(P_i^4)),
\end{equation}
where the $(P_i^0,\dots ,P_i^3)$ are the simplices of the triangulation of $N$ lifted to $\widetilde{N}$. 
This works as well replacing the triangulation by any cycle representing the fundamental class $[N,\partial N]$. 

From an ideal triangulation of $M$ as in \cite{Bergeron_Falbel_Guilloux}, 
where degenerate tetrahedra -- meaning tetrahedra contained in planes -- are allowed, 
we obtain a relative cycle representing the fundamental class $[N,\partial N]$ 
by triangulating the intersection of the ideal triangulation with $\widetilde{N}$. 
The formula \eqref{eq:interpr} now follows. 

Finally, a simple cohomological argument using the naturality of the transfer maps $\hb^3(\Gamma)\rightarrow \hcb^3(\PSL(2,\mathbb{C}))$ 
allows us to reinterpret our Borel invariant in terms of a multiplicative constant in {bf Proposition~\ref{prop: beta as cst} 
of Section~\ref{sec:characteristic_number}}. 
It is this interpretation of the Borel invariant which we will use for the proof of our main Theorem~\ref{thm: main}  in Section~\ref{sec:4}.

\section{The cocycle representing $\beta_\mathrm{b}(n)$}\label{new section def cocycle}
\subsection{Configuration spaces}\label{sec:conf_sp}

For $k,m\geq 0$, let $\sigma_k(m)$ be the quotient of the set of all spanning 
$(k+1)$-tuples $(x_0,\dots,x_k)$ of vectors in $\mathbb{C}^m$ by the diagonal action of $\mathrm{GL}(m,\mathbb{C})$. 
Observe that for $k+1<m$, the set $\sigma_k(m)$ is empty. 

Given an $m$-dimensional complex vector space $V$ and a $(k+1)$-tuple $(x_0,\dots,x_k)$ of vectors spanning $V$, 
we obtain by choosing an isomorphism $V\rightarrow \mathbb{C}^m$ a well defined element of $\sigma_k(m)$ denoted $[V;(x_0,\dots,x_k)]$. 

On
$$\sigma_k:=\sqcup_{m\geq 0} \sigma_k(m)=\sigma_k(0)\sqcup \dots \sqcup \sigma_k(k+1),$$
there are two kinds of face maps
$$\varepsilon_i^{(k)}, \ \eta_i^{(k)}:\sigma_k\longrightarrow \sigma_{k-1},$$
for $0\leq i\leq k$, given as
$$ \varepsilon_i^{(k)}([\mathbb{C}^m;(x_0,\dots,x_k)])=[\langle x_0,\dots,\widehat{x_i}\dots,x_k\rangle; ( x_0,\dots,\widehat{x_i},\dots,x_k)],$$
$$\eta_i^{(k)}([\mathbb{C}^m;(x_0,\dots,x_k)])=[\mathbb{C}^m/ \langle x_i \rangle;(x_0,\dots,\widehat{x_i},\dots,x_k)],$$
where in the last definition, $x_j$ is understood as the image of $x_j$ in $\mathbb{C}^m/ \langle x_i \rangle$ and $\langle y_0,\dots,y_\ell\rangle$ 
denotes the linear subspace spanned by the set $\{ y_0,\dots,y_\ell\}$. 
Observe that on $\sigma_k(m)$ both face maps take values in $\sigma_{k-1}(m)\sqcup\sigma_{k-1}(m-1)$.

One can easily verify the following relations between these face maps; for all $0\leq j<i\leq k$:
\begin{align}
\label{(R1)}\varepsilon_j^{(k-1)} \varepsilon_i^{(k)}&=\varepsilon_{i-1}^{(k-1)} \varepsilon_{j}^{(k)}, \tag{R1}\\
\label{(R2)} \eta_j^{(k-1)} \eta_i^{(k)}&=\eta_{i-1}^{(k-1)} \eta_{j}^{(k)},\tag{R2}\\
\label{(R3)} \eta_j^{(k-1)} \varepsilon_i^{(k)}&=\varepsilon_{i-1}^{(k-1)} \eta_{j}^{(k)}\tag{R3}.
\end{align}

Let us denote, for $k\geq 0$, by $\Z[\sigma_k]$ the free abelian group on $\sigma_k$ and set $\Z[\sigma_k]=0$ for $k\leq -1$. 
We extend the face maps to morphisms $\Z[\sigma_k]\to \Z[\sigma_{k-1}]$. For $k\geq 1$, 
we define $\partial_k,d_k$ and $D_k:\Z[\sigma_k]\to Z[\sigma_{k-1}]$ by 
$$\partial_k( \tau ):=\sum_{i=0}^k (-1)^i\varepsilon_i^{(k)}( \tau ),$$
$$d_k( \tau ):=\sum_{i=0}^k (-1)^i\eta_i^{(k)}( \tau ),$$
for $\tau \in \sigma_k$, and extend this definition to $\partial_k=d_k=0$ for $k\leq 0$. 
Finally, we set $D_k:=\partial_k-d_k$, for any $k\in \mathbb{Z}$. 

From the relations (\ref{(R1)}-\ref{(R3)}), we immediately deduce:

\begin{lem}\label{lem: DD=0}  Let $k\in \mathbb{Z}$. Then
\begin{itemize}
\item $\partial_{k-1}\partial_k=0$,
\item $d_{k-1}d_k=0$,
\item $\partial_{k-1}d_k+d_{k-1}\partial_k=0$,
\item $D_{k-1}D_k=0$. 
\end{itemize}
\end{lem}

We have thus established that $(\Z[\sigma_k],D_k)$ is a chain complex.

Observe that the symmetric group $S_{k+1}$ acts on $\sigma_k(m)$ and 
hence on $\sigma_k$. We let 
\begin{equation*}
\R_\mathrm{alt}(\sigma_k)=\{f:\sigma_k\to \R \mid f \mathrm{\ is \ alternating \ w.r.t \ the \ } S_{k+1}\mathrm{-action}\},
\end{equation*} 
and let $D_k^*:\R(\sigma_{k-1})\to \R(\sigma_k)$ denote the dual of $D_k\otimes_\R \mathds{1}:\R[\sigma_k]\to \R[\sigma_{k-1}]$. 
Then we obtain from Lemma~\ref{lem: DD=0}:

\begin{lem} $(\R_\mathrm{alt}(\sigma_k),D_k^*)$ is a cochain complex.
\end{lem}

\subsection{The volume cocycle}\label{section: Def cocycle}\label{sec:volume_cocycle}

The signed hyperbolic volume $\Vol_{\mathbb{H}^3}:(\partial \mathbb{H}^3)^4=\mathbb{P}(\mathbb{C}^2)^4\to \R$ 
of ideal simplices in hyperbolic $3$-space $\mathbb{H}^3$, 
which we consider as a function on $(\mathbb{C}^2\smallsetminus \{0\})^4$ in the obvious way,  extends to a function
$$\Vol:\sigma_3\longrightarrow \R$$
by setting  $\Vol|_{\sigma_3(m)}=0$ for all $m\neq 2$ and
$$\Vol([\mathbb{C}^2;v_0,\dots,v_3]):=\left\{ \begin{array}{ll}
\Vol_{\mathbb{H}^3}(v_0,\dots,v_3) & \quad \mbox{if }v_{i}\neq0\mbox{ for all }i,\\
0 &\quad \mathrm{otherwise.} \end{array}\right.$$
\begin{thm}\label{thm: Df=0} The function $\Vol\in \R_\mathrm{alt}(\sigma_3)$ satisfies $D_4^* \Vol=0$.
\end{thm}

\begin{proof}
First observe that $D^*_4
\Vol(\tau)=\Vol(D_4\tau)=0$ if $\tau\in \sigma_4(0)\sqcup \sigma_4(1)\sqcup \sigma_4(4)\sqcup \sigma_4(5)$. 
Thus we have two cases to consider, namely $\tau \in \sigma_4(2)$ and $\tau\in \sigma_4(3)$. 

\begin{lem}\label{lem2} Let $\tau\in \sigma_4(2)$. Then $\Vol(D_4\tau)=0$.
\end{lem}

\begin{proof} Let $\tau=[\mathbb{C}^2;v_0,\dots,v_4]$, then
\begin{equation*}
\begin{aligned}
D^*_4\Vol(\tau)=&\sum_{i=0}^4 (-1)^i \Vol([\langle v_0,\dots,\widehat{v_i},\dots,v_4 \rangle;v_0,\dots,\widehat{v_i},\dots,v_4])\\
-&\sum_{i=0}^4 (-1)^i \Vol ([\mathbb{C}^2/\langle v_i \rangle ;v_0,\dots,\widehat{v_i},\dots,v_4]).
\end{aligned}
\end{equation*}
Suppose that $v_j\neq 0$ for every $j$. 
Observe that whether $v_0,\dots,\widehat{v_i},\dots,v_4 $ generate $\mathbb{C}^2$ or not, we have
\begin{equation*}
\Vol([\langle v_0,\dots,\widehat{v_i},\dots,v_4 \rangle,v_0,\dots,\widehat{v_i},\dots,v_4])=\Vol_{\mathbb{H}^3}(v_0,\dots,\widehat{v_i},\dots,v_4).
\end{equation*} 
Indeed, if the $v_i$'s do generate $\mathbb{C}^2$, this is the definition of $\Vol$, 
and if not, then on the one hand $\Vol([\langle v_0,\dots,\widehat{v_i},\dots,v_4 \rangle,v_0,\dots,\widehat{v_i},\dots,v_4])=0$ 
and on the other hand $\Vol_{\mathbb{H}^3}(v_0,\dots,\widehat{v_i},\dots,v_4)=0$ 
since the projection of $v_j$'s in $\partial \mathbb{H}^3=  \mathbb{P}(\mathbb{C}^2)$ coincide. 
As all the $v_j\neq 0$ and $\Vol_{\mathbb{H}^3}$ is a cocycle on $\partial \mathbb{H}^3=  \mathbb{P}(\mathbb{C}^2)$, the first sum vanishes. 
So does each term of the second sum since $\mathbb{C}^2/\langle v_i \rangle$ is $1$-dimensional. 

If $v_j=0$ for some $j$, then for every $i\neq j$, the $i$-th term of the first and second sums vanish
since the null vector appears in the argument each time. 
It follows that $D^*_4\Vol(\tau)$ is equal to
\begin{equation*}
\begin{aligned}
(-1)^j\big\{ & \Vol(\underbrace{[\langle v_0,\dots,\widehat{v_j},\dots,v_4 \rangle}_{\mathbb{C}^2};v_0,\dots,\widehat{v_j},\dots,v_4])\\
- &\Vol ([\underbrace{\mathbb{C}^2/\langle v_j \rangle}_{\mathbb{C}^2} ;v_0,\dots,\widehat{v_j},\dots,v_4])\big\}=0,
\end{aligned}
\end{equation*}
which finishes the proof of the lemma.\qed
\end{proof}

\begin{lem}\label{lem3}Let $\tau\in \sigma_4(3)$. Then $\Vol(D_4\tau)=0$.
\end{lem}

\begin{proof} Let $\tau=[\mathbb{C}^3;v_0,\dots,v_4]$, then
\begin{equation*}
\begin{aligned}
D^*_4\Vol(\tau)
=&\sum_{i=0}^4 (-1)^i \Vol([\langle v_0,\dots,\widehat{v_i},\dots,v_4 \rangle;v_0,\dots,\widehat{v_i},\dots,v_4])\\
-&\sum_{i=0}^4 (-1)^i \Vol ([\mathbb{C}^3/\langle v_i \rangle ;v_0,\dots,\widehat{v_i},\dots,v_4]).
\end{aligned}
\end{equation*}

We distinguish several cases:
\begin{description}
\item [{\bf $v_j=0$ for some $j$:}] For every $i\neq j$, the $i$-th term of the first and second sums vanish. 
The $j$-th terms are also both zero 
since both spaces $\langle v_0,\dots,\widehat{v_j},\dots,v_4 \rangle$ and $\mathbb{C}^3/\langle v_j \rangle$ are $3$- and not $2$-dimensional.
\end{description}
From now on we suppose that $v_j\neq 0$ for every $j$.
\begin{description}
\item[{\bf $\langle v_i\rangle=\langle v_j\rangle$ for some pair $i\neq j$:}] By alternation we can suppose that $i=0$, $j=1$. 
In this case, $\langle v_1,v_2,v_3,v_4\rangle=\langle v_0,v_2,v_3,v_4\rangle=\mathbb{C}^3$ 
and the two first terms of the first sum vanish. 
Since $\langle v_0\rangle=\langle v_1\rangle$ the three last terms of the first sum vanish also. 
In the second sum, the two first terms vanish since the null vector appears in the argument, 
while the last three terms vanish since either $v_0$ and $v_1$ are zero in the corresponding quotients or they span the same line. 
\end{description}
We suppose from now on that all lines generated by the $v_i$'s are distinct.
\begin{description}
\item[{\bf $\mathrm{Dim}(\langle v_0,\dots,\widehat{ v_j},\dots, v_4\rangle )=2$ for some $j$:}] By alternation we can suppose that $j=4$. 
Since no two vectors lie on the same line, 
it follows that $\langle v_0,\dots,\widehat{ v_i},\dots, v_3\rangle \cong \mathbb{C}^2$ for every $i\in \{0,\dots,3\}$;
since the $5$-tuple generates $\mathbb{C}^3$, we further get $\langle v_0,\dots,\widehat{ v_i},\dots, v_3,v_4\rangle = \mathbb{C}^3$. 
The two sums for $D^*_4\Vol(\tau)$ thus reduce to
$$\Vol([\langle v_0,\dots,v_3 \rangle;v_0,\dots,v_3])-\sum_{i=0}^4 (-1)^i \Vol ([\mathbb{C}^3/\langle v_i \rangle ;v_0,\dots,\widehat{v_i},\dots,v_4]).$$
The composition of injection and projection
$$\langle v_0,\dots,v_3\rangle \hookrightarrow \mathbb{C}^3 \longrightarrow \mathbb{C}^3/\langle v_4\rangle$$
gives an isomorphism identifying
$$[\langle v_0,\dots,v_3 \rangle;v_0,\dots,v_3]=[\mathbb{C}^3/\langle v_4 \rangle;  v_0,\dots,v_3 ].$$
Thus, the remaining term for the first sum cancels with the last term of the second sum. We are left with
$$-\sum_{i=0}^3 (-1)^i \Vol ([\mathbb{C}^3/\langle v_i \rangle ;v_0,\dots,\widehat{v_i},\dots,v_3,v_4]).$$
But since $\langle v_0,\dots,\widehat{ v_i},\dots, v_3\rangle=\langle v_0,v_1,v_2,v_3\rangle$ the projections of $v_j$, 
for any $i\neq j$ between $0$ and $3$ generate the same complex line in the quotient $\mathbb{C}^3/\langle v_i \rangle$ 
so that $\Vol$ vanishes on such configurations, finishing the proof in this case. 
\item[{\bf For every $j$, $\langle v_0,\dots,\widehat{ v_j},\dots, v_4\rangle =\mathbb{C}^3$ and $(v_0,\dots,v_4)$ is not generic:}] 
Recall that a $q$-tuple of vectors $(w_1,\dots,w_q)$ in $\mathbb{C}^m$ is generic 
if and only if $\mathrm{Dim}\langle w_{i_1},\dots,w_{i_k}\rangle =k$ whenever $k\leq m$ and the $1\leq i_1,\dots, i_k\leq q$ are distinct. 
As we have assumed that none of the $v_j$ vanish and no $4$-subtuple generate $\mathbb{C}^2$, 
the only way this $5$-tuple can be non-generic is if $3$ among the vectors generate a $2$-dimensional subspace. 
We can without loss of generality suppose that $\mathrm{Dim}(\langle v_2,v_3,v_4\rangle )=2$. 
Since $\langle v_0,\dots,\widehat{ v_j},\dots, v_4\rangle =\mathbb{C}^3$ for every $j$, the first sum vanishes. 
As in the previous case, the images of $v_3$ and $v_4$ generate the same line in $\mathbb{C}^3/\langle v_2 \rangle$; 
the analogous statement holds for   $\mathbb{C}^3/\langle v_3 \rangle$ and  $\mathbb{C}^3/\langle v_4 \rangle$, 
so that the $i$-th term of the second sum vanishes for $i=2,3,4$. 
Finally, we have
$$\mathbb{C}^3=\langle v_2,v_3,v_4\rangle \oplus \langle v_1\rangle=\langle v_2,v_3,v_4\rangle \oplus \langle v_0 \rangle.$$
Since there exists $g\in \mathrm{GL}(3,\mathbb{C})$ fixing the plane $\langle v_2,v_3,v_4 \rangle$ and sending $v_1$ to $v_0$ it follows that 
$$\Vol( [\mathbb{C}^3/\langle v_0 \rangle; v_1,v_2,v_3,v_4] )=\Vol( [\mathbb{C}^3/\langle v_1 \rangle; v_0,v_2,v_3,v_4] ),$$
finishing the proof of this case.
\item[{\bf $(v_0,\dots,v_4)$ is generic:}] As in the previous case, the first sum cancels since the spaces in the arguments are all $3$-dimensional. 
Using the surjective linear maps $f_i$ of the next Lemma~\ref{lem: LinAlg}, we have
$$[\mathbb{C}^3/\langle v_i\rangle; v_0,\dots, \widehat{v_i},\dots, v_4]=[\mathbb{C}^2; f_i(v_0),\dots,f_i( \widehat{v_i}),\dots, f_i(v_4)].$$
Then Lemma~\ref{lem: LinAlg} below implies that $w_k:=\langle f_i(v_k)\rangle$ for $i\neq k$ is independent of $i$,
so that the second sum now reduces to 
$\sum_{i=0}^4 (-1)^i \Vol_{\mathbb{H}^3}(w_0,\dots,\widehat{w_i},\dots,w_4)$, for $w_j\in f_i(v_j)$ for $i\neq j$, 
which is equal to $0$ since $\Vol_{\mathbb{H}^3}$ is a cocycle.
\end{description} \qed
\end{proof}

\begin{lem}\label{lem: LinAlg}
Let $v_0,\dots,v_4$ be a generic $5$-tuple in $\mathbb{C}^3$ and let $L_i=\langle v_i\rangle$. 
Then there are $5$ lines $\ell_0,\dots,\ell_4$ in $\mathbb{C}^2$ and 
surjective linear maps $f_i:\mathbb{C}^3\rightarrow \mathbb{C}^2$ with $f_i(L_i)=0$ and $f_i(L_j)=\ell_j$ for $j\neq i$. 
\end{lem}

\begin{proof} We use $\{v_2,v_3,v_4\}$ as a basis of $\mathbb{C}^3$ and express everything in those coordinates. 
Then $v_0=(a_2,a_3,a_4)$ and $v_1=(b_2,b_3,b_4)$ with all coordinates nonzero. Let
$$g=\left( \begin{array}{ccc}
a_2/b_2 & 0 & 0\\
0 & a_3/b_3&0\\
0&0& a_4/b_4
\end{array}
\right)\in \mathrm{GL}(3,\mathbb{C}),$$
so that $g(v_1)=v_0$ and $g(v_i)\in \mathbb{C}^* v_i$ for $i\in \{2,3,4\}$. 
Choose a surjection$f_0:\mathbb{C}^3\rightarrow \mathbb{C}^2$ with kernel $L_0$ and define  
$$\ell_0=f_0\circ g(L_0) ,\quad \ell_1=f_0(L_1),\quad \ell_2=f_0(L_2),\quad \ell_3=f_0(L_3),\quad \ell_4=f_0(L_4).$$
Then $f_0$ automatically satisfies the conclusion of the lemma, and so does $f_1=f_0\circ g$. 
We proceed to construct $f_i$ for $i=2,3,4$. 
By symmetry, it is enough to construct $f_2$. Set $f_2=\mu_0 f_0 +\mu_1 f_1$, with $\mu_0,\mu_1\neq 0$ chosen such that $f_2(v_2)=0$. 
Note that such $\mu_0$ and $\mu_1$ exist since $f_0(v_2)$ and $f_1(v_2)$ belong to the same line $\ell_2\smallsetminus \{0\}$
and do not vanish. 
Since for $i=3,4$ $f_0(v_i)$ and $f_1(v_i)$ both belong to $\ell_i$, , 
then $f_2(v_i)$ belongs to $\ell_i$ as well. 
To show that $f(L_i)=\ell_i$ it is thus enough to show that $f_2(v_i)$, say $f_2(v_3)\neq 0$. 
Suppose the contrary, then together with $f_2(v_2)=0$ we get
\begin{eqnarray*}
\mu_0a_2+\mu_1 b_2&=&0,\\
\mu_0a_3+\mu_1 b_3&=&0,
\end{eqnarray*}
and hence $b_2/a_2=b_3/a_3$. But this implies $b_2v_0-a_2v_1=(b_2a_4-a_2b_4)v_4$ 
which is a contradiction with the assumption that $v_0,\dots,v_4$ is generic. 
Also, $f_2(v_0)=\mu_0 f_0(v_0)+\mu_1 f_1(v_0)=\mu_1f(v_0)\in \ell_1\smallsetminus\{0\}$ 
so that indeed $f(L_0)=\ell_0$, and similarly $f(L_1)=\ell_1$, finishing the proof of the lemma.
\qed
\end{proof}

Putting together Lemma~\ref{lem2} and \ref{lem3} finishes the proof of Theorem~\ref{thm: Df=0}.
\qed
\end{proof}

\section{Affine flags}\label{sec:affine_flags}
A complete flag $F$ in $\mathbb{C}^n$ is a sequence of $(n+1)$-vector subspaces $F^0\subset F^1\subset \dots \subset F^n$ with $\mathrm{dim}(F^j)=j$. 
An affine flag $(F,v)$ is a pair consisting of a flag $F$ and an $n$-tuple of vectors $v=v^1,v^2,\dots,v^n$ such that
$$F^j=\mathbb{C} v^j+F^{j-1}, \ \ j\geq 1.$$
The group $\mathrm{GL}(n,\mathbb{C})$ acts naturally on the space $\mathcal{F}(\mathbb{C}^n)$ of all flags and the space $\mathcal{F}_\mathrm{aff}(\mathbb{C}^n)$ of affine flags.

We consider $(\mathbb{Z}[\mathcal{F}_\mathrm{aff}(\mathbb{C}^n)^{k+1}],\partial_k)$, 
where $\mathbb{Z}[\mathcal{F}_\mathrm{aff}(\mathbb{C}^n)^{k+1}]$ is the free abelian group on $\mathcal{F}_\mathrm{aff}(\mathbb{C}^n)^{k+1}$ 
and $\partial_k$ is the boundary map associated to the usual face maps 
$\varepsilon_i^{(k)}:\mathcal{F}_\mathrm{aff}(\mathbb{C}^n)^{k+1}\rightarrow \mathcal{F}_\mathrm{aff}(\mathbb{C}^n)^{k}$ 
consisting in dropping the $i$-th component for $k\geq 0$. 
We extend the definition to $\partial_0:\Z[\mathcal{F}_\mathrm{aff}(\mathbb{C}^n)]\rightarrow 0$. 
Our aim is to construct a morphism of complexes, or almost so,
$$T_k:(\Z[\mathcal{F}_\mathrm{aff}(\mathbb{C}^n)^{k+1}],\partial_k)\longrightarrow(\Z[\sigma_k],D_k).$$
To this end, for every multiindex $\J \in [0,n-1]^{k+1}$, we define the map 
$$t_\J:\mathcal{F}_\mathrm{aff}(\mathbb{C}^n)^{k+1}\rightarrow \sigma_k$$ 
by
$$t_\J((F_0,v_0),\dots,(F_k,v_k)):= \left[  \frac{\langle F_0^{j_0+1},\dots , F_k^{j_k+1}\rangle }
{\langle F_0^{j_0},\dots , F_k^{j_k}\rangle};(v_0^{j_0+1},\dots,v_k^{j_k+1})\right]$$
and finally $T_k:\Z[\mathcal{F}_\mathrm{aff}(\mathbb{C}^n)^{k+1}]\rightarrow \Z[\sigma_k]$ by 
$$T_k((F_0,v_0),\dots,(F_k,v_k))=\sum_{\J\in [0,n-1]^{k+1}} t_\J((F_0,v_0),\dots,(F_k,v_k)).$$

\begin{lem} \label{lem: almost commute} Let $k\geq 1$. We have:
\begin{enumerate}
\item If $k$ is odd, $T_{k-1}\partial_k-D_kT_k=0$.
\item If $k$ is even, $T_{k-1}\partial_k-D_kT_k$ evaluated on an affine flag equals
$$n^k [0;\underbrace{(0,\dots,0)}_{k}]\in \Z[\sigma_{k-1}(0)].$$
\end{enumerate}
\end{lem}

\begin{proof} One verifies the following relations for every $0\leq i\leq k$ and $\J\in [0,n-1]^{k+1}$:
\begin{enumerate}
\item[(a)] If $j_i\leq n-2$, then $\eta_i^{(k)} t_\J=\varepsilon_i^{(k)} t_{\J+\delta_i}$, where $\delta_i=(0,\dots,0,1,0,\dots,0)$ with the $1$ in the $i$-th position.
\item[(b)] If $j_i=n-1$, then $\eta_i^{(k)} t_\J((F_0,v_0),\dots,(F_k,v_k))=[0;(0,\dots,0)]$.
\item[(c)] If $j_i=0$, then $\varepsilon_i^{(k)}t_\J=T_{\J(i)} \varepsilon_i^{(k)}$, where $\J(i)\in [0,n-1]^k$ is obtained from $\J$ by dropping $j_i$. 
\end{enumerate}
We evaluate 
\begin{equation*}
\begin{aligned}
D_k T_k((F_0,v_0),\dots,(F_k,v_k))
=\sum_{i=0}^k (-1)^i \bigg(&\sum_\J \varepsilon_i^{(k)} t_\J((F_0,v_0),\dots,(F_k,v_k))\\
-&\sum_\J \eta_i^{(k)} t_\J((F_0,v_0),\dots,(F_k,v_k))\bigg).
\end{aligned}
\end{equation*}
Splitting the first inner sum into a sum over $\J\in [0,n-1]^{k+1}$ with $j_i=0$ and $\J$ with $j_i\geq 1$, 
we obtain using (c) from the first contribution the value $T_{k-1}\varepsilon_i^{(k)}((F_j,v_j))$ 
while the second contribution together with the second inner sum adds up to $-n^k[0;(0,\dots,0)]$ using (a) and (b).
\qed
\end{proof}

Now we dualize the objects so far considered, as in Section~\ref{sec:conf_sp}. 
For the natural $S_{k+1}$-action on $\mathcal{F}_\mathrm{aff}(\mathbb{C}^n)^{k+1}$, 
the spaces $\R_\mathrm{alt}(\mathcal{F}_\mathrm{aff}(\mathbb{C}^n)^{k+1})$ of alternating cochains 
together with the dual $\partial_k^*$ of $\partial_k\otimes_\R \mathds{1}$ form a complex. 
Denoting $T^*_k$ the dual of $T_k\otimes_\R \mathds{1}$ 
we obtain immediately from Lemma~\ref{lem: almost commute}:

\begin{prop} \label{prop: T cochain map} The map 
$T^*_k:\R_\mathrm{alt}(\sigma_k)\to \R_\mathrm{alt}(\mathcal{F}_\mathrm{aff}(\mathbb{C}^n)^{k+1})$ 
is a morphism of complexes, 
taking values in the subcomplex 
$\R_\mathrm{alt}(\mathcal{F}_\mathrm{aff}(\mathbb{C}^n)^{k+1})^{\GL(n,\mathbb{C})}$ of $\GL(n,\mathbb{C})$-invariants.
\end{prop}

In particular, defining now
\begin{equation}\label{eq:cocycle}
\begin{aligned}
B_n((F_0,v_0),\dots,(F_3,v_3)):=&T^*_3 \Vol((F_0,v_0),\dots,(F_3,v_3))\\
=&\sum_{\J\in[0,n-1]^4}\Vol\left(\left[\frac{\langle F_0^{j_0+1},\dots,F_3^{j_3+1}\rangle}{\langle F_0^{j_0},\dots,F_3^{j_3}\rangle};(v_0^{j_0+1},\dots,v_3^{j_3+1})\right]\right),
\end{aligned}
\end{equation}
where $\Vol\in \R_\mathrm{alt}(\sigma_3)$ is the function on $\sigma_3$ defined in Section~\ref{sec:volume_cocycle} we obtain

\begin{cor} The function $B_n$ is a $\GL(n,\mathbb{C})$-invariant alternating cocycle 
defined on all $4$-tuples of affine flags in $\mathcal{F}_\mathrm{aff}(\mathbb{C}^n)^4$, 
which descends to a well defined $\GL(n,\mathbb{C})$- hence $\PGL(n,\mathbb{C})$-invariant function 
on the space $\mathcal{F}(\mathbb{C}^n)^4$ of $4$-tuples of flags.
\end{cor}

\begin{proof} That $B_n$ is alternating follows from the same property of $\Vol$. 
The fact that it is a cocycle follows from Proposition~\ref{prop: T cochain map} and 
Theorem~\ref{thm: Df=0}. Finally, it descends to  $\mathcal{F}(\mathbb{C}^n)^4$ 
since $\Vol([\mathbb{C}^2;v_0,\dots,v_4])$ only depends on the lines generated by $v_0,\dots,v_4$. 
\qed
\end{proof}


\section{Boundedness of $B_n$}\label{sec:boundedness}

The aim of this section is to establish the following

\begin{thm}
\label{thm: Volm bounded}
Let $F_{0},...,F_{3}\in\mathcal{F}(\mathbb{C}^{n})$
be arbitrary flags. Then 
\begin{equation*}
B_n(F_{0},\dots,F_{3})\leq \frac{1}{6}n(n^2-1)\cdot v_{3}.
\end{equation*}
\end{thm}

Recall that $v_3$ denotes the volume of the maximal ideal tetrahedron in $\mathbb{H}^3$. 
In the next section we will characterize the equality case, for which
it will be useful to know, as a preliminary case, that equality can
happen only if the flags are in general position, i.e. flags for which 
\begin{equation*}
\mbox{dim}\left(\langle F_{0}^{j_{0}},...,F_{3}^{j_{3}}\rangle\right)=j_{0}+...+j_{3}
\end{equation*}
whenever $j_{0}+...+j_{3}\leq n$. 
\begin{lem}
\label{lem: sharp implies general position}If equality holds in Theorem
\ref{thm: Volm bounded} then the flags $F_{0},...,F_{3}$ are in
general position. 
\end{lem}
We postpone the proof of Lemma~\ref{lem: sharp implies general position}
to after the proof of Theorem~\ref{thm: Volm bounded} 
as it uses some of the arguments therein.  
We start by introducing some useful notation and giving an argument for Theorem~\ref{thm: Volm bounded}
in the case in which the configuration of flags is generic.

Given four flags $F_0,\dots,F_3\in \mathcal{F}(\mathbb{C}^n)$ 
we denote by $\mathbb{F}=(F_0,\dots,F_3)$ the corresponding quadruple of flags. 
For any multi-index $\mathds{J}=(j_0,\dots,j_3)$ with $0\leq j_i\leq n-1$ 
we let $Q(\mathbb{F},\mathds{J})$ be the quotient 
$$
Q(\mathbb{F},\mathds{J}):=
\frac{\langle F_{0}^{j_{0+1}},...,F_{3}^{j_{3+1}}\rangle}{\langle F_{0}^{j_{0}},...,F_{3}^{j_{3}}\rangle}
$$
and denote by $f(\mathbb{F},\mathds{J})\subset Q(\mathbb{F},\mathds{J})$
the $4$-tuple of $0$ or $1$-dimensional subspaces obtained by projecting
$(F_{0}^{j_{0}+1},...,F_{3}^{j_{3}+1})$ to $Q(\mathbb{F},\mathds{J})$.

Furthermore, for any nonnegative integers $k,n$ with  $k\geq1$ we set 
\begin{equation*}
C_{k}(n)=\sharp   \left\{ (a_{1},\ldots,a_{k})\left|  a_{i}\in \mathbb{N},\ \sum_{i=1}^{k}a_{i}=n\right. \right\} .
\end{equation*}
Note that $C_{k}(0)=1$, $C_{k}(1)=k$, $C_{1}(n)=1$ and we have
the recursive relation 
\begin{equation}
C_{k}(n)=C_{k-1}(n)+C_{k}(n-1),\label{eq: Rec}
\end{equation}
for $k\geq 2$, $n\geq 1$. Indeed the set underlying $C_{k}(n)$ is
the disjoint union of the $k$-tuples with $a_{k}=0$ giving the term
$C_{k-1}(n)$ and the $k$-tuples with $a_{k}\geq1$ which is in bijection
with the set underlying $C_{k}(n-1)$ via $a_{k}\mapsto a_{k}-1$.
Using the relation (\ref{eq: Rec}) it is straightforward
to conclude that 
\begin{eqnarray*}
C_{4}(n) & = & \binom{n+1}{1}C_{3}(0)+\binom{n+1}{2}C_{2}(1)+\binom{n+1}{3}C_{1}(2)\\
 & = & \frac{1}{6}(n+1)(n+2)(n+3).
 \end{eqnarray*}

Observe that $C_{4}(n-2)$ is exactly the number of $\mathds{J}=(j_{0},\ldots,j_{3})$
with $0\leq j_{i}\leq n-2$ for which the dimension 
\begin{equation*}
\mbox{dim}\left(\QFJ \right)=2,
\end{equation*}
for quadruple of flags $\mathbb{F}=(F_{0},\ldots,F_{3})\in(\mathcal{F}(\mathbb{C}^{n}))^4$ in general
position. Let us outline the short argument assuming $n\geq3$.  We distinguish three cases:
\begin{itemize}
\item If $j_{0}+...+j_{3}\leq n-3$,
then the dimension of the vector space
$\langle F_{0}^{j_{0+1}},...,F_{3}^{j_{3+1}}\rangle$
is equal to the minimum between $n$ and $j_{0}+\dots+j_{3}+4$ so that
the quotient has dimension 
\begin{equation*}
\mbox{min}\left\{ n,j_{0}+\ldots +j_{3}+4\right\} -(j_{0}+\ldots+j_{3})\geq\mbox{min}\left\{ 3,4\right\} \geq3.
\end{equation*}
\item If $j_{0}+...+j_{3}=n-2$ then $\langle F_{0}^{j_{0+1}},...,F_{3}^{j_{3+1}}\rangle$
has dimension $n$ so that the quotient has dimension $2$. 
\item If $j_{0}+...+j_{3}\geq n-1$, then the quotient has dimension
$0$ or $1$.
\end{itemize}
We conclude that $C_{4}(n-2)$ is the number of nonzero summand
in $B_n$ for generic flags. This proves Theorem~\ref{thm: Volm bounded} for generic $4$-tuples of flags. 
The non-generic case is more involved and we start with the following simple observation.

\begin{lem}\label{lem: at most one j3 given j0 j1 j2}
Let $\mathbb{F}=(F_{0},\ldots,F_{3})\in\mathcal{F}(\mathbb{C}^{n})^4$
be arbitrary $4$-tuple of flags. Then for every $0\leq j_{0},j_{1},j_{2}\leq n-2$
there exists at most one $0\leq j_{3}\leq n-2$ such that 
\begin{equation*}
\Vol\left(Q(\mathbb{F},(j_0,\dots,j_3));f(\mathbb{F},(j_0,\dots,j_3)\right)\neq0.
\end{equation*}
\end{lem}
\begin{proof}
If there is no $j_3$ with $0\leq j_3\leq n-2$ such that $\mbox{dim}\left(\QFJ \right)=2$
and $F_{3}^{j_3+1}\neq0$ in $\QFJ$, we are done.
Otherwise take $j_3$ minimal satisfying these two conditions.
This implies that $F_{0}^{j_{0}+1},F_{1}^{j_{1}+1},F_{2}^{j_{2+1}}$
all lie on the same line in $\mathbb{C}^{n}/\langle F_{3}^{j_{3}+1}\rangle$
and hence in $\mathbb{C}^{n}/\langle F_{3}^{j}\rangle$ for any $j>j_{3}$
and also in $\langle F_{0}^{j_{0+1}},...,F_{3}^{j_{+1}}\rangle/\langle F_{0}^{j_{0}},...,F_{3}^{j}\rangle$.
Thus the volume vanishes for $j>j_{3}$.  It clearly vanishes for $j<j_3$ by definition of $j_3$.
\qed
\end{proof}
Notice that the lemma implies immediately that 
\begin{equation}
\left|\sum_{\mathds{J}\in \left\{ \begin{array}{c}
j_{0}=j_{1}=j_{2}=0,\\
0\leq j_{3}\leq n-2\end{array}\right\}}\mbox{Vol}\left(\QFJ;\fFJ\right)\right|\leq v_{3}\label{eq: 3 ji = 0 bounded by v3}
\end{equation}
and 
\begin{equation}
\left|\sum_{\mathds{J}\in \left\{\begin{array}{c}
j_{0}=j_{1}=0\\
0\leq j_{2},j_{3}\leq n-2\end{array}\right\}}\mbox{Vol}\left(\QFJ;\fFJ\right)\right|\leq C_{2}(n-2)\cdot v_{3};\label{eq: 2 ji = 0 bounded by C2(m-2)v3}
\end{equation}
in fact there are $C_{2}(n-2)=n-1$ choices for $j_{2}$ giving by Lemma
\ref{lem: at most one j3 given j0 j1 j2} each at most one nonzero
summand. We will further show:
\begin{lem}
\label{lem: 1 ji =00003D 0 bounded by C3(m-2)v3}Let $\mathbb{F}=(F_{0},\dots,F_{3})\in(\mathcal{F}(\mathbb{C}^{n}))^4$
be an arbitrary quadruple of flags. Then
\begin{equation*}
\left|\sum_{\mathds{J}\in \left\{\begin{array}{c}
j_{0}=0\\
0\leq j_{1},j_{2},j_{3}\leq n-2\end{array}\right\}}\mathrm{Vol}\left(\QFJ;\fFJ\right)\right|\leq C_{3}(n-2)\cdot v_{3}.
\end{equation*}
\end{lem}
\begin{proof}[of Theorem~\ref{thm: Volm bounded} and Lemma~\ref{lem: 1 ji =00003D 0 bounded by C3(m-2)v3}]
We prove the theorem and the lemma simultaneously by induction on $n$. 

For $n=2$ there is only one summand $(j_{0},...,j_{3})=(0,...,0)$
in both the theorem and the lemma, so the inequalities are immediate.
Suppose that the theorem and the lemma are proven for $n-1$. By definition,
we have 
\begin{equation*}
B_n(F_{0},\dots,F_{4})=\sum_{\mathds{J}\in \{ ( j_{0},\dots,j_{3}) \mid 0\leq j_i\leq n-2\}}\mbox{Vol}\left(\QFJ;\fFJ\right).
\end{equation*}
Indeed if $j_i=n-1$ then the quotient is $1$ or $0$-dimensional and the volume vanishes. 
We split the above sum into three, summing over
\begin{eqnarray*}
J_1&=&\{( j_{0},\dots,j_{3}) \mid j_0=j_1=0,\ 0\leq j_2,j_3\leq n-2\}, \\
J_2&=&\{( j_{0},\dots,j_{3}) \mid j_0=0, \ 0<j_1\leq n-2,\ 0\leq j_2,j_3\leq n-2\}, \\
J_3&=&\{( j_{0},\dots,j_{3}) \mid 0<j_0\leq n-2, 0\leq j_1,j_2,j_3\leq n-2\}. 
\end{eqnarray*}

We first analyze the sum over $J_3$. Denote by $\overline{V}$ the image of a subspace $V\subset\mathbb{C}^{n}$
under the projection onto $\mathbb{C}^{n}/\langle F_{0}^{1}\rangle.$ If
$F\in\mathcal{F}(\mathbb{C}^{n})$ is a complete flag, then we denote
by $\overline{F}\in\mathcal{F}(\mathbb{C}^{n}/\langle F_{0}^{1}\rangle)$
the complete flag we obtain as the projection of $F$. More precisely,
the $n+1$ subspaces of $F$ project onto $n$ distinct subspaces
in the quotient, giving the $n$ distinct subspaces of $\overline{F}$:
\begin{equation*}
\underbrace{\overline{F^{0}}}_{=\overline{F}^{0}}\subset...\subset\underbrace{\overline{F^{i-1}}}_{=\overline{F}^{i-1}}\subset\underbrace{\overline{F^{i}}=\overline{F^{i+1}}}_{=\overline{F}^{i}}\subset\underbrace{\overline{F^{i+2}}}_{=\overline{F}^{i+1}}\subset...\subset\underbrace{\overline{F^{n}}}_{=\overline{F}^{n-1}}=\mathbb{C}^{n}/\langle F_{0}^{1}\rangle,
\end{equation*}
for some $0\leq i\leq n-1$. Note in particular that $\overline{F}^{j}$
is equal to either $\overline{F^{j}}$ or $\overline{F^{j+1}}$ (or
both in the unique case of $j=i$). The projection of $F_{0}$ is
\begin{equation*}
\overline{F_{0}^{0}}=\overline{F_{0}^{1}}\subset\overline{F_{0}^{2}}\subset...\subset\overline{F_{0}^{n}}=\mathbb{C}^{n}/\langle F_{0}^{1}\rangle,
\end{equation*}
so in this case, the $j$-th space of $\overline{F_{0}}$ is always
the projection of the $(j+1)$-th space of $F_{0}$.

Note that since in the sum over $J_3$, the index $j_{0}$ is greater or equal to $1$, 
we have for each summand $1\leq j_{0}\leq n-2$, $0\leq j_{1},j_{2},j_{3}\leq n-2$
that 
\begin{equation*}
\mbox{Vol}\left(\QFJ;\fFJ\right)
=\mbox{Vol}\left(\frac{\langle\overline{F_{0}^{j_{0+1}}},...,\overline{F_{3}^{j_{3+1}}}\rangle}{\langle\overline{F_{0}^{j_{0}}},...,\overline{F_{3}^{j_{3}}}\rangle};\overline{F_{0}^{j_{0}+1}},...,\overline{F_{3}^{j_{3}+1}}\right).
\end{equation*}
Furthermore, if $j_{k}$, for $1\leq k\leq3$ is such that $\overline{F_{k}^{j_{k}}}=\overline{F_{k}^{j_{k}+1}}$
then the space $\overline{F_{k}^{j_{k}+1}}$ is $0$ in the quotient
and thus the volume vanishes. Instead of summing on the
dimensions of the spaces of the flags $F_{0},...,F_{3}$, we can thus
sum over the dimensions of the spaces of the quotient flags $\overline{F_{0}},...,\overline{F_{3}}$
and the  sum over $J_3$ becomes 
\begin{equation}
\begin{aligned}
\sum_{0\leq i_{0},...,i_{3}\leq n-3}&\mbox{Vol}\left(\frac{\langle\overline{F_{0}}^{i_{0+1}},...,\overline{F_{3}}^{i_{3}+1}\rangle}{\langle\overline{F_{0}}^{i_{0}},...,\overline{F_{3}}^{i_{3}}\rangle};\overline{F_{0}}^{i_{0}+1},...,\overline{F_{3}}^{i_{3+1}}\right)\\
&\leq C_{4}(n-3)\cdot v_{3},\label{eq:2nd sum rewritten in quotient}
\end{aligned}
\end{equation}
where the last inequality follows by the induction hypothesis in the theorem. 

Similarly, for the sum over $J_2$ 
we quotient by $F_{1}^{1}$ to obtain
\begin{equation*}
\begin{aligned}
\sum_{\mathds{J}\in \left\{\begin{array}{c}
j_{0}=0,\ 1\leq j_{1}\leq n-2\\
0\leq j_{1},j_{2},j_{3}\leq n-2\end{array}\right\}}&\mbox{Vol}\left(\QFJ;\fFJ\right)\\
=\sum_{\begin{array}{c}
i_{0}=0\\
0\leq i_{1},i_{2},i_{3}\leq n-3\end{array}}
&\mbox{Vol}\left(
\frac{\langle\overline{F_{0}}^{i_{0+1}},...,\overline{F_{3}}^{i_{3}+1}\rangle}
{\langle\overline{F_{0}}^{i_{0}},...,\overline{F_{3}}^{i_{3}}\rangle};\overline{F_{0}}^{i_{0}+1},...,\overline{F_{3}}^{i_{3+1}}\right)\\
\leq C_{3}(n-3)\cdot v_{3},\hphantom{XXXX}&
\end{aligned}
\end{equation*}
by the induction hypothesis in the lemma. Since the sum over $J_1$
is by (\ref{eq: 2 ji = 0 bounded by C2(m-2)v3}) bounded by
$C_{2}(n-2)\cdot v_{3}$, it follows that the sum over $J_1$ and $J_2$
is indeed bounded by 
\begin{equation*}
\left(C_{2}(n-2)+C_{3}(n-3)\right)\cdot v_{3}=C_{3}(n-2)\cdot v_{3},
\end{equation*}
which proves the lemma. It follows that $B_n(F_{0},\dots,F_{3})$ is bounded by
\begin{equation*}
\left(C_{3}(n-2)+C_{4}(n-3)\right)\cdot v_{3}=C_{4}(n-2)\cdot v_{3},
\end{equation*}
which proves the theorem.
\qed 
\end{proof}

\begin{proof}[of Lemma~\ref{lem: sharp implies general position}]
We prove
the lemma by induction on $n$. For $n=2$ the four flags $F_{0},...,F_{3}$
(which are given by their lines $F_{i}^{1}$) are in general position
if and only if $\mbox{dim}(\langle F_{i}^{1},F_{j}^{1}\rangle)=2$
for every $i\neq j$, i.e. if and only if the lines are distinct.
But if two lines are equal, then $B_2(F_{0},\ldots,F_{3})=0$.
Suppose the lemma proven for $n-1$. As before, denote by $\overline{F}$
the projection of a complete flag $F\in\mathcal{F}(\mathbb{C}^{n})$
to a complete flag in $\mathcal{F}(\mathbb{C}^{n}/\langle F_{0}^{1}\rangle)$.
Recall that in particular, $\overline{F^{j}}=\overline{F}^{j}$ or
$\overline{F}^{j-1}$. Suppose that $j$ is minimal such that $\overline{F_{1}^{j}}=\overline{F_{1}}^{j-1}$
or equivalently such that $F_{0}^{1}\subset F_{1}^{j}$. By the proof
of Theorem~\ref{thm: Volm bounded}, $B_{n}(F_{0},\ldots,F_{3})$
is maximal if and only if each of the sums over $J_1,J_2$ and $J_3$ is
maximal. In particular, by symmetry, the sum over $j_0=j_2=0$ is also maximal and hence
\begin{equation*}
\begin{aligned}
\sum_{\begin{array}{c}
j_{0}=j_{2}=0\\
0\leq j_{1},j_{3}\leq m-2\end{array}}
&\mbox{Vol}\left(\frac{\langle F_{0}^{1},F_{1}^{j_{1}+1},F_{2}^{1},F_{3}^{j_{3+1}}\rangle}{\langle F_{1}^{j_{1}},F_{3}^{j_{3}}\rangle};F_{0}^{1},F_{1}^{j_{1}+1},F_{2}^{1},F_{3}^{j_{3}+1}\right)\\
=&C_{2}(n-2)\cdot v_{3}.
\end{aligned}
\end{equation*}
But for $j_{1}\geq j$, the space $F_{0}^{1}$ is $0$ in the quotient
$\frac{\langle F_{0}^{1},F_{1}^{j_{1}+1},F_{2}^{1},F_{3}^{j_{3+1}}\rangle}{\langle F_{1}^{j_{1}},F_{3}^{j_{3}}\rangle}$,
while for $j_{1}=j-1$, the spaces $F_{0}^{1}$ and $F_{1}^{j}$ are
equal in the quotient. In both cases, the volume vanishes.
By Lemma~\ref{lem: at most one j3 given j0 j1 j2} it follows that
the above sum is smaller or equal to $(j-1)\cdot v_{3},$ hence we must
have $j-1\geq C_{2}(n-2)=n-1$. 

We have thus established that $F_{0}^{1}\subset F_{1}^{n}\smallsetminus F_{1}^{n-1}$and
by symmetry the same holds for the flags $F_{2}$ and $F_{3}$. In
particular, $\overline{F_{k}^{j}}=\overline{F_{k}}^{j}$ for $k=1,2,3$
and $0\leq j\leq n-1$, while $\overline{F_{0}^{j}}=\overline{F_{0}}^{j-1}$
for $j\geq1$. 

Let $0\leq j_{0},j_{1},j_{2},j_{3}\leq n$ be such that $j_{0}+...+j_{3}\leq n$.
Since the case $j_{0}=...=j_{3}=0$ is trivial we can by symmetry
suppose that $j_{0}\geq1$. Again, it follows from the proof of the
theorem that $B_{n}(F_{0},\ldots,F_{3})$ is maximal if and
only if the sum over $J_3$ is maximal.
This sum is rewritten in (\ref{eq:2nd sum rewritten in quotient})
as 
\begin{equation*}
B_{n-1}(\overline{F_{0}},....,\overline{F_{3}})=C_{4}(n-3)\cdot v_{3}.
\end{equation*}
Thus by induction, the flags $\overline{F_{0}},...,\overline{F_{3}}$
are in general position. It remains to compute 
\begin{eqnarray*}
\mbox{dim}(\langle F_{0}^{j_{0}},...,F_{3}^{j_{3}}\rangle) & = & \mbox{dim}(\langle\overline{F_{0}^{j_{0}}},...,\overline{F_{3}^{j_{3}}}\rangle)+1\\
 & = & \mbox{dim}(\langle\overline{F_{0}}^{j_{0}-1},\overline{F_{1}}^{j_{1}},...,\overline{F_{3}}^{j_{3}}\rangle)+1\\
 & = & (j_{0}-1)+j_{1}+j_{2}+j_{3}+1=j_{0}+j_{1}+j_{2}+j_{3},
 \end{eqnarray*}
which finishes the proof of the lemma. 
\qed
\end{proof}

\section{Maximality properties of the cocycle}\label{sec: max}

The main result of this section is Theorem~\ref{thm: max iff image of reg}, in which we characterize the configuration of $4$-tuples of flags on which $B_n$ is maximal.
This configuration is related to the Veronese embedding to which we now turn. 
The irreducible representation $\pi_n:\mathrm{PSL}(2,\mathbb{C})\rightarrow \mathrm{PSL}(n,\mathbb{C})$ induces a $\pi_n$-equivariant boundary map 
\begin{equation*}
\varphi_n:\mathbb{P}(\mathbb{C}^2)\longrightarrow\mathcal{F}(\mathbb{C}^{n}),
\end{equation*}
also known as the Veronese embedding. It is defined as follows: $\varphi_n\left(\left[\begin{array}{c}
x\\
y\end{array}\right]\right)$ is the complete flag with $(n-1)$-dimensional space with basis \begin{eqnarray*}
 & \left\{ \left(\begin{array}{c}
x\\
y\\
0\\
\vdots\\
\\0\end{array}\right),\left(\begin{array}{c}
0\\
x\\
y\\
0\\
\vdots\\
0\end{array}\right),...,\left(\begin{array}{c}
0\\
\vdots\\
\\0\\
x\\
y\end{array}\right)\right\} ,\end{eqnarray*}
where $\left[\begin{matrix}x\\y\end{matrix}\right]$ are homogeneous coordinates on $\mathbb{P}(\mathbb{C}^2)$.
The lower dimensional spaces are then obtained inductively with basis
$v_{i}^{'}=xv_{i}+yv_{i+1}$, for $i=1,...,k-1$, where $\{v_{1},...,v_{k}\}$
is the basis of the $k$-dimensional space. More precisely, the basis of the $(n-2)$-dimensional
space is 
\begin{equation*}
\left\{  \left(\begin{array}{c}
x^{2}\\
2xy\\
y^{2}\\
0\\
\vdots\\
0\end{array}\right),...,\left(\begin{array}{c}
0\\
\vdots\\
0\\
x^{2}\\
2xy\\
y^{2}\end{array}\right)\right\} .
\end{equation*}
The $(n-i)$-dimensional space has as basis the vectors 
\begin{equation}\label{eq:basis}
\left( \underbrace{0,...,0}_{k},x^{i},\binom{i}{1}x^{i-1}y^{1},\dots,\binom{i}{j}x^{i-j}y^{j},\dots,\binom{i}{i-1}x^{1}y^{i-1},y^{i},\underbrace{0,\dots,0}_{n-i-k-1}\right)^{T},
\end{equation}
for $k=0,\dots,n-i-1$. We give another useful description of this complete
flag. For $i=1,\dots,n-1$, set 
\begin{equation*}
z_{i}^{n}=\left(x^{i},\ \binom{i}{1}x^{i-1}y,\ \binom{i}{2}x^{i-2}y^{2},\ ...\ ,\binom{i}{i}y^{i},\underbrace{0,\dots,0}_{n-i-1}\right)^{T}.
\end{equation*}
Note that $z_{i}^{n}$ is the first vector of the basis in \eqref{eq:basis}
of the $(n-i)$-th space of 
$\varphi_n\left(\left[\begin{array}{c}
x\\
y\end{array}\right]\right)$, that is the one corresponding to $k=0$.
Furthermore, since $z_{i}^{n}$ does not belong to the space generated
by $z_{n-1}^{n},z_{n-2}^{n},\dots,z_{i+1}^{n}$, the $(n-i)$-th space
admits the alternative basis 
\begin{equation*}
\left\{ z_{n-1}^{n},z_{n-2}^{n},\dots,z_{i+1}^{n},z_{i}^{n}\right\} .
\end{equation*}
With this at hand, it is easy to prove the following:
\begin{lem}
\label{lem: projection of maximal is maximal}Let $D$ be the $(n-1)\times(n-1)$
diagonal matrix with diagonal entries $1,2,\dots,n-1$. Let $p$ be
the projection $\mathbb{C}^{n}\rightarrow\mathbb{C}^{n-1}\cong\langle e_{2},...,e_{n}\rangle$
with kernel $\langle e_{1}\rangle$. Then 
\begin{equation*}
D\cdot p\left(\varphi_n\left(\left[\begin{array}{c}
x\\
y\end{array}\right]\right)\right)=\varphi_{n-1}\left(\left[\begin{array}{c}
x\\
y\end{array}\right]\right),
\end{equation*}
for any $\left[\begin{array}{c}
x\\
y\end{array}\right]\in \mathbb{P}(\mathbb{C}^2)$.
\end{lem}

Note that the projection $p$ induces a map from the set of complete flags
in $\mathbb{C}^{n}$ to the set of complete flags in $\mathbb{C}^{n-1}$.
(See the proof of Theorem~\ref{thm: Volm bounded} and Lemma~\ref{lem: 1 ji =00003D 0 bounded by C3(m-2)v3} for a detailed description of the induced map.) 
\begin{proof}
For $y=0$, the complete flag $\varphi_n\left(\left[\begin{array}{c}
1\\
0\end{array}\right]\right)$ is given as \begin{equation*}
\langle e_{1}\rangle\subset\langle e_{1},e_{2}\rangle\subset\langle e_{1},e_{2},e_{3}\rangle\subset...\subset\langle e_{1},e_{2},...,e_{n-1}\rangle.\end{equation*}
Its projection by $p$ is the complete flag 
\begin{equation*}
\langle e_{2}\rangle\subset\langle e_{2},e_{3}\rangle\subset...\subset\langle e_{2},...,e_{n-1}\rangle.
\end{equation*}
Multiplication by $D$ leaves this complete flag invariant, and this is indeed the image of 
$\left[\begin{array}{c}
1\\
0\end{array}\right]$ under $\varphi_
{n-1}$. 

Suppose now that $y\neq0$.  We may assume that $y=1$. 
Note that the projection by $p$ of the $(n-i)$-dimensional space of 
$\varphi_
n\left(\left[\begin{array}{c}
x\\
y\end{array}\right]\right)$ 
is $(n-i)$-dimensional, so $p(z_{n-1}^{n}),\dots,p(z_{i}^{n})$ is a basis
of it. We show that $D\cdot p(z_{i}^{n})=i\cdot z_{i-1}^{n-1}$, from
which the lemma follows immediately. Indeed, projection by $p$ erases
the first entry of $z_{i}^{n}$. For the remaining entries, we have
that the $j$-th entry of $z_{i}^{n}$, for $2\leq j\leq n$ is 
\begin{equation*}
\binom{i}{j}x^{i-j}.
\end{equation*}
Multiplication by $D$ will multiply this entry by $j$, giving 
\begin{equation*}
i\cdot\binom{i-1}{j-1}x^{i-j},
\end{equation*}
which is precisely $i$ times the $(j-1)$-th entry of $z_{i-1}^{n-1}$.
\end{proof}

\begin{thm}
\label{thm: max iff image of reg} Let $F_{0},...,F_{3}\in\mathcal{F}(\mathbb{C}^{n})$.
Then 
\begin{equation*}
B_n(F_{0},\dots,F_{3})=\frac{1}{6}n(n^2-1)\cdot v_{3}
\end{equation*}
if and only if there exists $g\in\GL({n},\mathbb{C})$ and a
positively oriented regular simplex with vertices $\xi_{0},...,\xi_{3}\in \mathbb{P}(\mathbb{C}^2)$
such that 
\begin{equation*}
F_{i}=g(\varphi_n(\xi_{i})),
\end{equation*}
for $i=0,...,3$.\end{thm}
\begin{cor}
\label{cor:max4}
Let $F_{0},...,F_{3}\in\mathcal{F}(\mathbb{C}^{n})$ be a maximal $4$-tuple, in the sense that 
\begin{equation*}
|B_{n}(F_{0},\dots,F_{3})|=\frac{1}{6}n(n^2-1)\cdot v_{3}.
\end{equation*}
If for $F\in \mathcal{F}(\mathbb{C}^{n})$, there is equality
$$B_n(F_{0},\dots,F_2,F_{3})=B_n(F_{0},\dots,F_{2},F),$$
then $F=F_3$.
\end{cor}
For the rest of this section, we will use the notation introduced after Lemma~\ref{lem: sharp implies general position} 
at the beginning of Section~\ref{sec:boundedness}. 
The first direction of Theorem~\ref{thm: max iff image of reg} will follow from the following more general computation:
\begin{prop}\label{prop: value on irr}Let $\xi_{0},...,\xi_{3}\in \mathbb{P}(\mathbb{C}^2)$
and set $F_{i}=\varphi_n(\xi_{i})$.
Then 
\begin{equation*}
B_n(F_{0},\dots,F_{3})=\frac{1}{6}n(n^2-1)\cdot\Vol_{\mathbb{H}^3}(\xi_0,\dots,\xi_3).
\end{equation*}
\end{prop}
To prove Proposition~\ref{prop: value on irr} by induction, we first prove:
\begin{lem}
\label{lem: Rec for <=00003D of maximality}Let $\xi_{0},...,\xi_{3}\in \mathbb{P}(\mathbb{C}^2)$
and set $F_{i}=\varphi_n(\xi_{i})$.
Then 
\begin{equation*}
\begin{aligned}
\sum_{\begin{array}{c}
j_{0}=j_{1}=0\\
0\leq j_{2},j_{3}\leq n-2\end{array}}&\mbox{Vol}\left(\frac{\langle F_{0}^{1},F_{1}^{1},F_{2}^{j_{2}+1},F_{3}^{j_{3+1}}\rangle}
{\langle F_{2}^{j_{2}},F_{3}^{j_{3}}\rangle};F_{0}^{1},F_{1}^{1},F_{2}^{j_{2}+1},F_{3}^{j_{3}+1}\right)\\
=&C_{2}(n-2)\cdot \Vol_{\mathbb{H}^3}(\xi_0,\dots,\xi_3).
\end{aligned}
\end{equation*}
\end{lem}
\begin{proof}
Let $\xi_{0},...,\xi_{3}\in \mathbb{P}(\mathbb{C}^2)$.  If $\xi_i=\xi_j$ for $i\neq j$ then both sides of the equality vanish. 
By transitivity of $\mbox{SL}_{2}\mathbb{C}$ on distinct triples of points, it is enough to prove the lemma for the four points 
\begin{equation*}
\xi_{0}=\left[\begin{array}{c}
1\\
1\end{array}\right],\quad\xi_{1}=\left[\begin{array}{c}
z\\
1\end{array}\right],\quad\xi_{2}=\left[\begin{array}{c}
0\\
1\end{array}\right],\quad\xi_{3}=\left[\begin{array}{c}
1\\
0\end{array}\right],
\end{equation*}
where $z\in \mathbb{C}$. Then the line of the flag $\varphi_n(\xi_{0})$ is generated by the vector 
\begin{equation*}
\left(1,\binom{n-1}{1},\binom{n-1}{2},...,\binom{n-1}{n-1},1\right)^T
\end{equation*}
and the the line of the flag $\varphi_n(\xi_{1})$ is generated
by the vector 
\begin{equation*}
\left(z^{n-1},\binom{n-1}{1}z^{n-2},\binom{n-1}{2}z^{n-3},...,\binom{n-1}{n-1}z,1\right)^T.
\end{equation*}
The flag $\varphi_n(\xi_{2})$ is 
\begin{equation*}
\langle e_{1}\rangle\subset\langle e_{1},e_{2}\rangle\subset\langle e_{1},e_{2},e_{3}\rangle\subset...\subset\langle e_{1},e_{2},...,e_{n-1}\rangle
\end{equation*}
and the flag $\varphi_n(\xi_{3})$ is 
\begin{equation*}
\langle e_{n}\rangle\subset\langle e_{n},e_{n-1}\rangle\subset\langle e_{n},e_{n-1},e_{n-2}\rangle\subset...\subset\langle e_{n},e_{n-1},...,e_{2}\rangle.
\end{equation*}
The quotient $\langle F_{0}^{1},F_{1}^{1},F_{2}^{j_{2}+1},F_{3}^{j_{3+1}}\rangle/\langle F_{2}^{j_{2}},F_{3}^{j_{3}}\rangle$
can only be $2$-dimensional if $j_{2}+j_{3}=n-2$. Fix $0\leq j_{2}\leq n-2$.
and notice that there are exactly $C_{2}(n-2)=n-1$ such $j_{2}$'s.
Let $j_{3}=n-2-j_{2}$. The space generated by $\varphi_n(\xi_{2})^{j_{2}}$
and $\varphi_n(\xi_{3})^{n-2-j_{2}}$ is the space 
\begin{equation*}
\langle\varphi_n(\xi_{2})^{j_{2}},\varphi_n(\xi_{3})^{n-2-j_{2}}\rangle=\langle e_{1},...,e_{j_{2}},e_{j_{2}+2},...,e_{n}\rangle.
\end{equation*}
We choose as isomorphism between 
\begin{equation*}
\mathbb{C}^{n}/\langle\varphi_n(\xi_{2})^{j_{2}},\varphi_n(\xi_{3})^{n-2-j_{2}}\rangle\cong\mathbb{C}^{2}\cong\langle e_{j_{2}+1},e_{j_{2}+2}\rangle
\end{equation*}
the map which is induced by the orthogonal projection from $\mathbb{C}^{n}$
onto $\langle e_{j_{2}+1},e_{j_{2}+2}\rangle$. 
Then the four points defined by $F_{0}^{1},F_{1}^{1},F_{2}^{j_{2}+1},F_{3}^{j_{3}+1}$
in the projectivisation of the quotient  are
\begin{equation*}
\begin{aligned}
&\left[\binom{n-1}{j_{2}},\binom{n-1}{j_{2}+1}\right],\\ 
&\left[\binom{n-1}{j_{2}}z^{n-(j_{2}+1)},\binom{n-1}{j_{2}+1}z^{n-(j_{2}+2)}\right],\\
&\left[1,0\right]\text{ and }\left[0,1\right].
\end{aligned}
\end{equation*}
Acting with the diagonal $2$ by $2$ matrix with entries $\binom{n-1}{j_{2}}^{-1},\binom{n-1}{j_{2}+1}^{-1}$,
and rescaling the second vector by $z^{-n+(j_{2}+2)}$, the four
points become
\begin{equation*}
\left[1,1\right],\quad\left[z,1\right],\quad\left[1,0\right],\quad\left[0,1\right],
\end{equation*}
which are exactly the original vertices $\xi_0,\dots,\xi_3$. It
follows that
%
%
\begin{equation*}
\mbox{Vol}\left(\frac{\langle F_{0}^{1},F_{1}^{1},F_{2}^{j_{2}+1},F_{3}^{n-1-j_{2}}\rangle}
{\langle F_{2}^{j_{2}},F_{3}^{n-2-j_{2}}\rangle};F_{0}^{1},F_{1}^{1},F_{2}^{j_{2}+1},F_{3}^{n-1-j_{2}}\right)
=\Vol_{\mathbb{H}^3}(\xi_0,\dots,\xi_3),
\end{equation*}
which proves the lemma. 
\qed 
\end{proof}

\begin{proof}[of Proposition~\ref{prop: value on irr}] 
We prove the proposition by induction on $n$, establishing first the cases
of $n=2$ and $n=3$. For $n=2$, there is nothing to prove. For $n=3$, let $\xi_{0},...,\xi_{3}\in \mathbb{P}(\mathbb{C}^2)$. The volume
$B_3(\varphi_3(\xi_{0}),...,\varphi_3(\xi_{3}))$
is written as a sum over $0\leq j_{0},...,j_{3}\leq1$. For $(j_{0},...,j_{3})=(0,...,0)$
the quotient is $3$-dimensional so the summand is 0. We thus have
at most four nonzero summands given by letting one of the $j_{k}$'s
be equal to $1$. The set $\{(0,0,1,0),(0,0,0,1)\}$ is exactly the
set summed over in Lemma~\ref{lem: Rec for <=00003D of maximality}
for $n=3$, so the value of the volume on these two multi-indices is
equal to $2\cdot \Vol_{\mathbb{H}^3}(\xi_0,\dots,\xi_3)$. By symmetry, the same holds for $\{(1,0,0,0),(0,1,0,0)\},$so
that the value of $B_{3}(\varphi_n(\xi_{0}),...,\varphi_n(\xi_{3}))$
is indeed $4\cdot \Vol_{\mathbb{H}^3}(\xi_0,\dots,\xi_3)=C_{4}(1)\cdot \Vol_{\mathbb{H}^3}(\xi_0,\dots,\xi_3)$.

Suppose that $n\geq 4$ and let $\xi_{0},...,\xi_{3}\in \mathbb{P}(\mathbb{C}^2)$. As usual, the volume $B_n(\varphi_n(\xi_{0}),...,\varphi_n(\xi_{3}))$
is written as a sum over $0\leq j_{0},...,j_{3}\leq n-2$. We rewrite
this sum as a sum over the three sets
\begin{eqnarray*}& \{ 1\leq j_{0}\leq n-2,\,0\leq j_{1},j_{2},j_{3}\leq n-2\}, \\
&\{ 1\leq j_{1}\leq n-2,\,0\leq j_{0},j_{2},j_{3}\leq n-2\}, \\
&\{ j_{0}=j_{1}=0,\, 0\leq j_{2},j_{3}\leq n-2\}
\end{eqnarray*}
minus the sum over
\begin{equation*}
\{1\leq j_{0},j_{1}\leq n-2,\,0\leq j_{2},j_{3}\leq n-2\}.
\end{equation*}
It follows from Lemma~\ref{lem: Rec for <=00003D of maximality} that
the third term is equal to $C_{2}(n-2)\cdot v_{3}$. Taking the quotient
by $F_{0}^{1}$, the first term becomes 
\begin{equation*}
B_{n-1}(\overline{\varphi_n(\xi_{0})},...,\overline{\varphi_n(\xi_{3})}).
\end{equation*}
But by Lemma~\ref{lem: projection of maximal is maximal}, and with $D$ therein,
\begin{equation*}
\overline{\varphi_n(\xi_{i})}=D^{-1}\varphi_{n-1}(\xi_{i}),
\end{equation*}
for $0\leq i\leq3$. In particular,
the first term of the sum can be rewritten as 
\begin{equation*}
B_{n-1}(\varphi_{n-1}(\xi_{0}),\dots,\varphi_{n-1}(\xi_{3})),
\end{equation*}
which is equal to $C_{4}(n-3)\cdot \Vol_{\mathbb{H}^3}(\xi_0,\dots,\xi_3)$ by induction. By symmetry,
the same holds for the second term of the sum. For the fourth (and
last) term, we take first the quotient by $F_{0}^{1}$ and then by
$F_{1}^{1}$, apply twice Lemma~\ref{lem: projection of maximal is maximal}
to conclude that it is equal to 
\begin{equation*}
B_{n-2}(\varphi_{n-2}(\xi_{0}),\dots, \varphi_{n-2}(\xi_{3})),
\end{equation*}
which by induction is equal to $C_{4}(n-4)\cdot \Vol_{\mathbb{H}^3}(\xi_0,\dots,\xi_3)$.

In conclusion, $B_n(\varphi_n(\xi_{0}),...,\varphi_n(\xi_{3}))$
is equal to $\Vol_{\mathbb{H}^3}(\xi_0,\dots,\xi_3)$ times 
\begin{equation*}
\begin{aligned}
&\,2\cdot C_{4}(n-3)+\underbrace{C_{2}(n-2)}_{=C_{3}(n-2)-C_{3}(n-3)}-C_{4}(n-4) \\
= &\,C_{4}(n-3)+C_{3}(n-2)-\underbrace{C_{3}(n-3)+\underbrace{C_{4}(n-3)-C_{4}(n-4)}_{=C_{3}(n-3)}}_{=0}\\
= &\,C_{4}(n-2),
 \end{aligned}
 \end{equation*}
which finishes the proof of the proposition. 
\qed 
\end{proof}
\begin{lem}
\label{lem: transitivity on FFL}The group $\GL(n,\mathbb{C})$
acts transitively on triples $(F_{0},F_{1},L)$, where $F_{0},F_{1}\in\mathcal{F}(\mathbb{C}^{n})$
and $L$ is a line in $\mathbb{C}^{n}$ such that $(F_0,F_1,L)$ is in general position, that is 
such that for every $0\leq j_{0},j_{1}\leq n$, 
\begin{eqnarray*}
\mathrm{dim}\langle F_{0}^{j_0},F_{1}^{j_1}\rangle&=&\mathrm{min}\{{j_{0}+j_{1},n}\},\\
\mathrm{dim}\langle F_{0}^{j_0},F_{1}^{j_1},L\rangle&=&\mathrm{min}\{j_{0}+j_{1}+1,n\}.
\end{eqnarray*}
\end{lem}
\begin{proof}
It is well known that $\mbox{GL}({n},\mathbb{C})$ acts transitively on the set of pairs of transverse flags.
As a result we may assume that 
\begin{equation*}
\begin{aligned}
F_{0}=&\langle e_{1}\rangle\subset\langle e_{1},e_{2}\rangle\subset...\subset\langle e_{1},e_{2},...,e_{n-1}\rangle,\\
F_{1}=&\langle e_{n}\rangle\subset\langle e_{n},e_{n-1}\rangle\subset...\subset\langle e_{n},e_{n-1},...,e_{2}\rangle.
\end{aligned}
\end{equation*}
Let $L=\langle v\rangle$; if $v_j=0$ for some $1\leq j\leq n$, 
then $\dim\langle F_0^{j-1},F_1^{n-j},L\rangle=n-1$, contradicting the genericity assumption.
Thus all the co-ordinates of $v$ are non-zero; 
the diagonal matrix $\operatorname{diag}(1/v_1,\dots,1/v_n)$ then stabilises $F_0,F_1$ and sends $v$ to $e_1+\dots+e_n$.\qed
\end{proof}
\begin{lem}
\label{lem: Rec for =00003D> of maximality}For any generic (in the sense of Lemma~\ref{lem: transitivity on FFL}) triple $(F_0,F_1,F_2^1)$, 
where $F_0,F_1\in \mathcal{F}(\mathbb{C}^n)$ and $F_2^1$ is a line, there exists a unique line
$F_{3}^{1}$ such that 
\begin{equation*}
\sum_{\begin{array}{c}
0\leq j_{0},j_{1}\leq n-2\\
j_{0}+j_{1}=n-2\end{array}}\mbox{Vol}\left(\frac{\mathbb{C}^{n}}{\langle F_{0}^{j_{0}},F_{1}^{j_{1}}\rangle};F_{0}^{j_{0}+1},F_{1}^{j_{1}+1},F_{2}^{1},F_{3}^{1}\right)=C_{2}(n-2)\cdot v_{3}.
\end{equation*}
\end{lem}
\begin{proof}
By Lemma~\ref{lem: transitivity on FFL} we may assume that 
\begin{equation*}
F_{0}=\langle e_{1}\rangle\subset\langle e_{1},e_{2}\rangle\subset...\subset\langle e_{1},e_{2},...,e_{n-1}\rangle,
\end{equation*}
\begin{equation*}
F_{1}=\langle e_{n}\rangle\subset\langle e_{n},e_{n-1}\rangle\subset...\subset\langle e_{n},e_{n-1},...,e_{2}\rangle
\end{equation*}
and 
\begin{equation*}
L=\langle e_{1}+...+e_{n}\rangle.
\end{equation*}
For $0\leq j_{0}\leq n-2$ let $j_{1}=n-2-j_{0}$. The space generated
by $F_{0}^{j_{0}}$ and $F_{1}^{n-2-j_{0}}$ is the space 
\begin{equation*}
\langle F_{0}^{j_{0}},F_{1}^{n-(j_{0}+1)}\rangle=\langle e_{1},...,e_{j_{0}},e_{j_{0}+2},...,e_{n}\rangle.
\end{equation*}
The orthogonal projection of $\mathbb{C}^{n}$ onto $\langle e_{j_{0}+1},e_{j_{0}+2}\rangle$
induces an isomorphism 
\begin{equation*}
\mathbb{C}^{n}/\langle F_{0}^{j_{0}},F_{1}^{j_{1}}\rangle\cong\langle e_{j_{0}+1},e_{j_{0}+2}\rangle\,.
\end{equation*}
Let $v=(v_{1},...,v_{n})^{T}$ be a generator of $F_{3}^{1}$. The
points $F_{0}^{j_{0}+1},F_{1}^{j_{1}+1},F_{2}^{1},F_{3}^{1}$ are
mapped, in the projectivization of $\langle e_{j_{0}+1},e_{j_{0}+2}\rangle$,
to 
\begin{equation*}
\left[1,0\right],\quad\left[0,1\right],\quad\left[1,1\right],\quad\left[v_{j_{0}+1},v_{j_{0}+2}\right].
\end{equation*}
For this $4$-tuple to be the vertices of a positively oriented regular
simplex, we need $v_{j_{0}+1}/v_{j_{0}+2}=\omega=e^{i\pi/3}$, for every $0\leq j_{0}\leq n-2$.
Thus, $F_{3}^{1}$ is generated by 
\begin{equation*}
(\omega^{n-1},\omega^{n-2},...,\omega,1)^{T},
\end{equation*}
which proves the lemma.\qed
\end{proof}

\begin{proof}[of Theorem~\ref{thm: max iff image of reg}] 
The first direction of the theorem follows from the more general Proposition~\ref{prop: value on irr}. 
For the other direction, fix a positively oriented simplex with vertices $\xi_{0},...,\xi_{3}\in \mathbb{P}(\mathbb{C}^2)$.
Let $F_{0},...,F_{3}$ be flags such that 
\begin{equation*}
B_n(F_{0},\ldots,F_{3})=C_{4}(n-2)\cdot v_{3}.
\end{equation*}
By Lemma~\ref{lem: sharp implies general position} this implies that
the flags are in general position. By the transitivity of $\mbox{GL}({n},\mathbb{C})$
on pairs of flags and $1$-dimensional space all in generic positions
established in Lemma~\ref{lem: transitivity on FFL}, we can assume
that $F_{0}=\varphi_n(\xi_{0})$, $F_{1}=\varphi_n(\xi_{1})$
and the $1$-dimensional space of $F_{2}$ is $F_{2}^{1}=\varphi_n(\xi_{2})^{1}$.
Maximality and genericity imply that 
\begin{equation*}
\mbox{Vol}\left(\frac{\langle F_{0}^{j_{0}+1},F_{1}^{j_{1}+1},F_{2}^{j_{2}+1},F_{3}^{j_{3}+1}\rangle}
{\langle F_{0}^{j_{0}},F_{1}^{j_{1}},F_{2}^{j_{2}},F_{3}^{j_{1}}\rangle};F_{0}^{j_{0}+1},F_{1}^{j_{1}+1},F_{2}^{j_{2}+1},F_{3}^{j_{3}+1}\right)=v_{3}
\end{equation*}
for any $j_{0}+...+j_{3}=n-2$. Thus it follows by Lemma~\ref{lem: Rec for =00003D> of maximality}
that $F_{3}^{1}$ is uniquely determined and since $\varphi_n(\xi_{3})^{1}$
by the other direction of the proof also satisfies the condition of
the lemma, it follows that $F_{3}^{1}=\varphi_n(\xi_{3})^{1}$. 

Inductively suppose that $F_{3}^{j}=\varphi_n(\xi_{3})^{j}$.
We will show that $F_{3}^{j+1}=\varphi_n(\xi_{3})^{j+1}$.
Indeed, look at the quotient $\mathbb{C}^{n}/F_{3}^{j}$. By the genericity
of $\varphi_n(\xi_{0}),...,\varphi_n(\xi_{3})$,
the projections $\overline{F_{0}},\overline{F_{1}}$ of the flags
$F_{0}$ and $F_{1}$ are still in general position and moreover,
the line $F_{2}^{1}$ projects to a line $\overline{F_{2}^{1}}$ in
general position with respect to $\overline{F_{0}},\overline{F_{1}}$.
Note that the complete flag $F_{3}$ projects to a complete flag $\overline{F_{3}}$
with $\overline{F_{3}}^{j_{3}}=\overline{F_{3}^{j_{3}+j+1}}$. Maximality
implies that the volume is equal to $v_{3}$ for the $4$-tuple $(j_{0},...,j_{3})=(k,n-k-j-2,0,j)$
for any $0\leq k\leq n-j-2$, which we can rewrite as 
\begin{equation*}
\mbox{Vol}\left(\frac{\langle\overline{F_{0}}^{j_{0}+1},\overline{F}_{1}^{j_{1}+1},\overline{F_{2}}^{1},\overline{F_{3}}^{1}\rangle}
{\langle\overline{F_{0}}^{j_{0}},\overline{F}_{1}^{j_{1}}\rangle};\overline{F_{0}}^{j_{0}+1},\overline{F}_{1}^{j_{1}+1},\overline{F_{2}}^{1},\overline{F_{3}}^{1}\right)=v_{3}.
\end{equation*}
Thus by Lemma~\ref{lem: Rec for =00003D> of maximality}, $\overline{F_{3}}^{1}=\overline{F_{3}^{j+1}}$
is completely determined by this maximality condition. In particular,
$F_{3}^{j+1}$ is completely determined by the fact that 
\begin{equation*}
\mbox{Vol}\left(\frac{\langle F_{0}^{j_{0}+1},F_{1}^{j_{1}+1},F_{2}^{1},F_{3}^{j+1}\rangle}
{\langle F_{0}^{j_{0}},F_{1}^{j_{1}},F^{j}\rangle};F_{0}^{j_{0}+1},F_{1}^{j_{1}+1},F_{2}^{1},F_{3}^{j+1}\right)=v_{3},
\end{equation*}
for any $0\leq j_{0}\leq n-j-2$ and $j_{1}=n-j_{0}-j-2$. Since $\varphi_n(\xi_{3})^{j+1}$
also satisfies this maximality condition by the other direction of
the theorem, it follows that $F_{3}^{j+1}=\varphi_n(\xi_{3})^{j+1}$.

We have thus shown that $F_{3}=\varphi_n(\xi_{3})$. By
symmetry, we can apply the same argument to show that $F_{2}=\varphi_n(\xi_{2})$,
which finishes the proof of the theorem.\qed
\end{proof}

\section{Proof of Theorem~\ref{thm 2 intro}}\label{sec:hcb3GLnC}

Recall that the space $\sigma_k(n)$ is the quotient of 
$\{(x_0,\dots,x_k)\in (\mathbb{C}^n)^{k+1}\mid \langle x_0,\dots,x_k\rangle =\mathbb{C}^n\}$ by the diagonal $\mathrm{GL}(n,\mathbb{C})$-action 
and is thus in a natural way a complex manifold of dimension $(k+1-n)\cdot n$. 
The symmetric group $S_{k+1}$ acts on $\sigma_k(n)$ and 
we let $\mathcal{B}^\infty_\mathrm{alt}(\sigma_k)$ denote 
the Banach space of bounded alternating Borel functions on $\sigma_k$. Together with $D_k^*$, 
the dual of $D_k\otimes_\R \mathds{1}:\R[\sigma_k]\rightarrow \R[\sigma_{k-1}]$, 
we obtain a complex of Banach spaces $(\mathcal{B}^\infty_\mathrm{alt}(\sigma_*),D_*)$.

Using Proposition~\ref{prop: T cochain map}, we deduce that the restriction of $T^*_k$ to the subcomplexes of bounded Borel functions 
gives a morphism of complexes
\begin{equation*}
T^*_k:\mathcal{B}^\infty_\mathrm{alt}(\sigma_k)\rightarrow \mathcal{B}^\infty_\mathrm{alt}(\mathcal{F}_\mathrm{aff}(\mathbb{C}^n)^{k+1})^{\GL(n,\mathbb{C})}.
\end{equation*}

Recall now that $(\mathcal{B}^\infty_\mathrm{alt}(\mathcal{F}_\mathrm{aff}(\mathbb{C}^n)^{*+1}),\partial_*)$ 
is a strong resolution of $\R$ by $\GL(n,\mathbb{C})$-Banach modules (see \cite{Burger_Iozzi_app}) and 
thus we have a canonical map $c_*$ from the cohomology of the complex of $\GL(n,\mathbb{C})$-invariants 
to the bounded continuous cohomology $\hcb^*(\GL(n,\mathbb{C}))$ of $\GL(n,\mathbb{C})$. 
(As announced in \S~\ref{sec:invariant}, we are dropping the explicit dependence on the coefficients.)
As a result, we obtain by composing $c_k$ with the map induced in cohomology by $T_k^*$ a map
\begin{equation*}
S^k(n):\h^k(\mathcal{B}^\infty_\mathrm{alt}(\sigma_*))\longrightarrow \hcb^k(\mathrm{GL}(n,\mathbb{C}),\R).
\end{equation*}

\begin{prop}  For $k\geq2$ the diagram
\begin{equation*}
\xymatrix{
&\hcb^k(\GL(n+1,\mathbb{C}))\ar[dd]^{}\\
\h^k(\mathcal{B}_\mathrm{alt}^\infty(\sigma_*))\ar[ur]^{S^k(n+1)}\ar[dr]_{S^k(n)}
&\\
&\hcb^k(\GL(n,\mathbb{C})),
}
\end{equation*}
where the vertical arrow is induced by the left corner injection, commutes. 
\end{prop}

\begin{proof} If $n\geq2$, let  $i_n:\mathbb{C}^n\hookrightarrow\mathbb{C}^{n+1}$ denote the embedding
\begin{equation*}
\begin{pmatrix} x_1\\\vdots\\x_n\end{pmatrix}\mapsto\begin{pmatrix} x_1\\\vdots\\x_n\\0\end{pmatrix}\,.
\end{equation*}
We define $i_n:\mathcal{F}_\mathrm{aff}(\mathbb{C}^n)\to\mathcal{F}_\mathrm{aff}(\mathbb{C}^{n+1})$ by  $i_n((F,v)):=(F',v')$,
where $F'^j:=i_n(F^j)$, $v'^j=i_n(v^j)$ and $v'^{n+1}=e_{n+1}$, for $0\leq j\leq n$.

Let now $\mathbb J\in[0,n]^{k+1}$ and 
$I=\{i:\,0\leq i\leq k\text{ such that }j_i=n\}$.

One verifies that if $I=\varnothing$, then $\mathbb J\in[0,n-1]^{k+1}$ and
\begin{equation}\label{eq:empty}
t_\mathbb J(i_n((F_j,v_j)))=t_\mathbb J((F_j,v_j))\,,
\end{equation}
while if $I\neq\varnothing$, then 
\begin{equation}\label{eq:not empty}
t_\mathbb J(i_n((F_j,v_j)))=[\mathbb{C};\delta_i^I]\,,
\end{equation}
where $\delta_i^I=1$ if $i\in I$ and $0$ otherwise.

We deduce from \eqref{eq:empty} and \eqref{eq:not empty} that $i_n$ induces for $k\geq2$
a commutative diagram of complexes
\begin{equation*}
\xymatrix{
&\mathcal{B}_\mathrm{alt}^\infty(\mathcal{F}_\mathrm{aff}(\mathbb{C}^{n+1})^{k+1})\ar[dd]^{i_n^\ast}\\
\mathcal{B}_\mathrm{alt}^\infty(\sigma_k)\ar[ur]^{T_k^\ast}\ar[dr]_{T_k^\ast}
&\\
&\mathcal{B}_\mathrm{alt}^\infty(\mathcal{F}_\mathrm{aff}(\mathbb{C}^n)^{k+1}).
}
\end{equation*}
Indeed, for $k\geq2$ alternating functions vanish on $[\mathbb{C};(\delta_i^I)]$.
On the other hand $i_n^\ast$ implements the restriction map in bounded cohomology associated 
to the left corner injection.
\qed
\end{proof}

\begin{proof}[of Theorem~\ref{thm: 2 intro}] We have $\beta_\mathrm{b}(n)=S^3(n)([\Vol])$, 
where $\Vol\in \mathcal{B}^\infty_\mathrm{alt}(\sigma_3)$ was defined in Section~\ref{sec:volume_cocycle}. 
The compatibility under the left corner injection then follows from the above proposition. 
Now $\hcb^3(\GL(2,\mathbb{C}))$ is one dimensional, generated by $\beta_\mathrm{b}(2)$. 
Thus we deduce that $\beta_\mathrm{b}(n)\neq 0$ and $\dim\hcb^3(\GL(n,\mathbb{C}))\geq 1$. 
We will conclude by using the stability results from Monod \cite{Monod}. For $n\geq 2$, the diagram of short exact sequence
\begin{equation*}
\xymatrix{
(1)\ar[r]
&\mathbb{C}^\times\ Id\ar@{^{(}->}[r]
&\mathrm{GL}(n,\mathbb{C})\ar@{->>}[r]
&\mathrm{PGL}(n,\mathbb{C})\ar[r]\ar[d]^\cong
&(1)\\
(1)\ar[r]
&\mu_n\ Id\ar@{^{(}->}[r]\ar@{^{(}->}[u]
&\mathrm{SL}(n,\mathbb{C})\ar@{->>}[r]\ar@{^{(}->}[u]
&\mathrm{PSL}(n,\mathbb{C})\ar[r]
&(1)
}
\end{equation*}
induces a diagram of isometric isomorphisms in bounded cohomology
\begin{equation}\label{eq:diagram}
\xymatrix{
\hcb^*(\mathrm{GL}(n,\mathbb{C}))\ar[d]_\cong
&\hcb^*(\mathrm{PGL}(n,\mathbb{C}))\ar[l]_\cong\ar@{=}[d]\\
\hcb^*(\mathrm{SL}(n,\mathbb{C}))
&\hcb^*(\mathrm{PSL}(n,\mathbb{C}))\,.\ar[l]_\cong
}
\end{equation}
Hence \cite[Theorem~1.1 and Proposition~3.4]{Monod}  can be rephrased by saying that for 
$0\leq q\leq  n$, the standard embedding 
$\mathrm{GL}(n,\mathbb{C})\hookrightarrow\mathrm{GL}(n+1,\mathbb{C})$
induces an isomorphism
\begin{equation*}
\xymatrix@1{
\hcb^q(\mathrm{GL}(n+1,\mathbb{C}))\ar[r]^-\cong
&\hcb^q(\mathrm{GL}(n,\mathbb{C}))
}
\end{equation*}
and an injection
\begin{equation*}
\xymatrix@1{
\hcb^q(\mathrm{GL}(q,\mathbb{C}))\,\,\ar@{^{(}->}[r]
&\hcb^q(\mathrm{GL}(q-1,\mathbb{C}))\,.
}
\end{equation*}
%

Applying this to $q=3$ we obtain that $\mathrm{dim}\hcb^3(\mathrm{GL}(n,\mathbb{C}))=1$, 
which proves the first part of Theorem~\ref{thm: 2 intro}. 
As for the second part, it follows from Section~\ref{sec:boundedness} that $\| \beta_\mathrm{b}(n)\|_\infty\leq (1/6)n(n^2 -1)v_3$. 
For the other inequality, let $\varphi_n:  \mathbb{P}(\mathbb{C}^2)\rightarrow \mathcal{F}(\mathbb{C}^n)$ be the Veronese embedding. 
Then
\begin{equation*}
B_n(\varphi_n(\xi_0),\dots,\varphi_n(\xi_3))=\frac{n(n^2-1)}{6}B_2(\xi_0,\dots,\xi_3)
\end{equation*}
by Proposition~\ref{prop: value on irr} and as a result, $T^*_n(\beta_\mathrm{b}(n))=\frac{n(n^2-1)}{6} \beta_\mathrm{b}(2)$. Since $\|\beta_\mathrm{b}(2)\|_\infty=v_3$, we deduce 
\begin{equation*}
\frac{n(n^2 -1)}{6}v_3=\|\pi^*_n(\beta_\mathrm{b}(n))\|_\infty\leq \| \beta_\mathrm{b}(n))\|_\infty,
\end{equation*}
which, using \eqref{eq:diagram}, concludes the proof of Theorem~\ref{thm: 2 intro}.
\qed
\end{proof}

\section{Proof of Theorem~\ref{theorem:maxrep}}
\subsection{The Borel invariant as a multiplicative constant}\label{sec:characteristic_number}

The aim of this subsection is to identify the Borel invariant $\mathcal{B}(\rho)$ 
as a multiplicative factor in the composition of certain bounded cohomology maps (Proposition~\ref{prop: beta as cst}) 
and to establish the simple direction of Theorem~\ref{theorem:maxrep} (Lemma~\ref{lem: beta(i)}). 
The proof is identical to the corresponding statement in \cite[Proposition 3.3]{BBI-Mostow} 
and is based on the existence of a natural transfer map 
\begin{equation*}
\xymatrix{ \hb^*(\Gamma) \ar[r]^-{\operatorname{trans}_\Gamma \ \ }
& \hcb^*(\PSL(2,\mathbb{C})) 
}
\end{equation*}
defined by integrating a $\Gamma$-invariant cocycle over $\Gamma\backslash G$ to make it $G$-invariant.

We refer the reader to \cite[Section 3.2]{BBI-Mostow} 
for the complete definitions and proofs.

\begin{prop}\label{prop: beta as cst}
Let $\Gamma $ be a lattice in $ \PSL(2,\mathbb{C})$ and $\rho:\Gamma\rightarrow \PSL(n,\mathbb{C})$
be a representation. The composition 
\begin{equation*}
\xymatrix{  \hcb^3(\PSL(n,\mathbb{C})) \ar[r]^-{\rho^\ast}
&\hb^3(\Gamma) \ar[r]^-{\operatorname{trans}_\Gamma}
&\hcb^3(\PSL(2,\mathbb{C}))\cong\mathbb{R} }
\end{equation*}
maps $\beta_\mathrm{b}{(n)}$ to $\frac{\mathcal{B}(\rho)}{\mathrm{Vol}(\Gamma \backslash \PSL(2,\mathbb{C}))}\beta_\mathrm{b}{(2)}$ and 
$$\left| \frac{\mathcal{B}(\rho)}{\mathrm{Vol}(\Gamma \backslash \PSL(2,\mathbb{C}))} \right|\leq \frac{1}{6}n(n^2-1).$$
\end{prop}

\begin{lem}\label{lem: beta(i)} Let $i:\Gamma\hookrightarrow \PSL(2,\mathbb{C})$
be a lattice embedding. Then 
\begin{equation*}
(\pi_n\circ i)^*\beta_\mathrm{b}{(n)}=\frac{1}{6}(n-1)n(n+1) \,\mathrm{Vol}(i(\Gamma)\backslash \PSL(2,\mathbb{C}))\,\beta_\mathrm{b}{(2)}\,.
\end{equation*}
\end{lem}

\begin{proof} Setting $\rho=\pi_n \circ i$ in Proposition~\ref{prop: beta as cst}, 
we see that the pullback $\rho^*:\hcb^3(\PSL(n,\mathbb{C}))\rightarrow \hb^3(\Gamma)$ factors through $\hcb^3(\PSL(2,\mathbb{C}))$. 
The composition of maps of the proposition thus becomes
\begin{equation*}
\xymatrix{ \hcb^3(\PSL(n,\mathbb{C})) \ar[r]^-{\pi^\ast_n}&\hcb^3(\PSL(2,\mathbb{C}))\ar[r]^{\quad \quad i^*} 
&\hb^{3}(\Gamma) \ar[r]^-{\operatorname{trans}_\Gamma}
&\hcb^{3}(\PSL(2,\mathbb{C}))\cong\mathbb{R}. }
\end{equation*}
The conclusion is immediate from the fact that 
$$
\pi_n^*(\beta_\mathrm{b}{(n)})=\frac{1}{6}(n-1)n(n+1)\cdot \beta_\mathrm{b}{(2)}
$$ 
(Theorem~\ref{thm: 2 intro} ) and that 
$\operatorname{trans}_\Gamma\circ i^*=\mathrm{Id}$.
\qed 
\end{proof}


\subsection{Proof of Theorem~\ref{theorem:maxrep}}\label{sec:4}
An essential aspect of bounded cohomology is that pullbacks of representations, 
for example $\rho:\Gamma\rightarrow \mathrm{PSL}(n,\mathbb{C})$ in our case, 
are implemented by boundary maps. 
Recall that by Furstenberg, given any representation of $\Gamma$ into $\mathrm{PSL}(n,\mathbb{C})$, 
there is always an equivariant measurable map 
\begin{equation*}
\varphi:   \mathbb{P}(\mathbb{C}^2)\longrightarrow M^1( \mathcal{F}(\mathbb{C}^n)),
\end{equation*}
where $M^1(\mathcal{F}(\mathbb{C}^n))$ denotes the space of probability measures on the flag space. 
More precisely, for every $\gamma\in \Gamma$ and almost every $\xi\in  \mathbb{P}(\mathbb{C}^2)$, we have
\begin{equation}\label{equ:phiequivar}
\varphi(i(\gamma)\cdot \xi)=\rho(\gamma)\cdot \varphi(\xi).
\end{equation}


The bounded cohomology groups $\hcb^3(\PSL(n,\mathbb{C}))$ and 
$\hb^3(\Gamma)$ can both be computed from the corresponding $L^\infty$ 
equivariant cochains on $\mathcal{F}(\mathbb{C}^n)$ and $\partial \mathbb{H}^3=\mathbb{P}(\mathbb{C}^2)$ respectively. The image of $\beta_\mathrm{b}{(n)}$ by
$\rho^*:\hcb^3(\PSL(n,\mathbb{C}))\rightarrow \hb^n(\Gamma)$
is represented at the cochain level by the pullback by $\varphi$, or more precisely, by the following cocycle: 
\begin{equation*} 
\begin{array}{rcl}
(\partial \mathbb{H}^3)^{4}&\longrightarrow & \mathbb{R} \\
(\xi_0,\dots,\xi_3)&\longmapsto &  \varphi(\xi_0)\otimes \dots \otimes \varphi(\xi_3) [B_n], 
\end{array}
\end{equation*}
where the last expression means that the cocycle $B_n$ is integrated 
with respect to the product of the four measures $ \varphi(\xi_0),\dots , \varphi(\xi_3)$, \cite{Burger_Iozzi_app}. 
It should however be noted that the pullback in bounded cohomology cannot be in general be implemented by boundary maps, 
unless the class to pull back can be represented by a strict invariant Borel cocycle
\cite{Burger_Iozzi_app}.

The further composition with the transfer map amounts 
to integrating the preceding cocycle over a fundamental domain for ${\Gamma\backslash \PSL(2,\mathbb{C})}$. 
In conclusion, since  $\operatorname{trans}_\Gamma \circ \rho^*( B_n)$ is by Proposition~\ref{prop: beta as cst} equal to  
$\frac{\mathcal{B}(\rho)}{\mathrm{Vol}(\Gamma\backslash\PSL(2,\C))}\cdot \beta_\mathrm{b}{(2)}$ and 
at the cohomology level there are no coboundaries in degree $3$ \cite{Bloch}, 
the map $\operatorname{trans}_\Gamma \circ \rho^*$ sends the cocycle  $B_n$ to 
$\frac{\mathcal{B}(\rho)}{\mathrm{Vol}(\Gamma\backslash\PSL(2,\C))}\mathrm{Vol}_{\mathbb{H}^3}$. 
Thus, for almost every $\xi_0,\dots,\xi_3\in \partial \mathbb{H}^3$, we have 
\begin{equation} \label{eq:formula1} 
\int_{\Gamma\backslash \PSL(2,\mathbb{C})} \varphi( g\xi_0)\otimes \dots \otimes \varphi(g\xi_3) [B_n]d\mu(\dot g)
=\frac{\mathcal{B}(\rho)}{\mathrm{Vol}(\Gamma\backslash\PSL(2,\C))}\mathrm{Vol}_{\mathbb{H}^3}(\xi_0,\dots,\xi_3)\,.
\end{equation}

It is however shown in \cite[Proposition~4.2 for $n=3$]{BBI-Mostow} that this almost everywhere equality is in fact a true equality.
As a consequence, we show that in the maximal case, the map $\varphi$ takes essentially values in the set of Dirac masses: 

\begin{cor}  Let $i:\Gamma\to \PSL(2,\mathbb{C})$ be a lattice embedding, 
$\rho:\Gamma\to \PSL(n,\mathbb{C})$ a representation and 
$\varphi:\partial \mathbb{H}^3\to M^1(\mathcal{F}(\mathbb{C}^n))$ a $\rho$-equivariant measurable map. 
Suppose that $|\mathcal{B}(\rho)|=\frac{1}{6}n(n^2-1)\cdot\mathrm{Vol}(\Gamma\backslash \PSL(2,\mathbb{C}))$. 
Then for almost every $\xi \in   \mathbb{P}(\mathbb{C}^2)$ the image $\varphi(\xi)$ is a Dirac mass.
\end{cor}

\begin{proof} Upon conjugating $\rho$ by the anti-holomorphic map $I$ induced by $z\mapsto \overline{z}$, 
we can without loss of generality assume that $\mathcal{B}(\rho)=\frac{1}{6}n(n^2-1)\cdot\mathrm{Vol}(\Gamma\backslash \PSL(2,\mathbb{C}))$. 
Assume $\mathrm{Vol}_{\mathbb{H}^3}(\xi_0,\dots,\xi_3)=v_3$. 
Then \eqref{eq:formula1} holds.  We deduce then from the fact that $|B_n|$ is bounded by $\frac{1}{6}n(n^2-1)v_3$ 
(see Theorem~\ref{thm: Volm bounded})
that 
%
\begin{equation*}
\varphi( g\xi_0)\otimes \dots \otimes \varphi(g\xi_3) [B_n]=\frac{1}{6}n(n^2-1)v_3
\end{equation*}
for almost every $g\in \mathrm{SL}(2,\mathbb{C})$. 
As a consequence, for almost every $(F_0,\dots,F_3)\in \mathcal{F}(\mathbb{C}^n)^4$ 
with respect to the product measure $\varphi(g\xi_0)\otimes \dots \otimes \varphi(g \xi_3)$, we have equality
\begin{equation*}
B_n(F_0,\dots,F_3)=\frac{1}{6}n(n^2-1)v_3.
\end{equation*}
Fix a triple $(F_0,F_1,F_2)$ such that the previous equality holds for $\varphi(g\xi_3)$- almost every $F_3$. 
However, by Corollary~\ref{cor:max4}, this $F_3$ is unique which implies that the support of $\varphi(g\xi_3)$ is reduced to one point. 
Since this holds for almost every $g\in \mathrm{SL}(2,\mathbb{C})$, the corollary is proven.
\qed
\end{proof}

If equality $|\mathcal{B}(\rho)|=  \frac{1}{6}n(n^2-1)\mathrm{Vol}(\Gamma\backslash \PSL(2,\mathbb{C}))$ holds, 
then upon conjugating $\rho$ by the anti-holomorphic map $I$ 
which has the effect of changing the sign of $\mathcal{B}(\rho)$ and composing $\varphi$ with the induced boundary map $I$, 
we can suppose that $\mathcal{B}(\rho)= \frac{1}{6}n(n^2-1)\mathrm{Vol}(\Gamma\backslash \PSL(2,\mathbb{C}))$. 
It then follows from the above that $\varphi$ maps 
almost every maximal $4$-tuples in $ P^1(\mathbb{C})$ to maximal $4$-tuples in $\mathcal{F}(\mathbb{C}^n)$.

\begin{theorem} \label{thm: step 3} Let $\varphi:  \mathbb{P}(\mathbb{C}^2)\rightarrow \mathcal{F}(\mathbb{C}^n)$ be a measurable map 
sending almost every maximal $4$-tuple in $  \mathbb{P}(\mathbb{C}^2)$ to a maximal $4$-tuple in $\mathcal{F}(\mathbb{C}^n)$. Then there exists $g\in \PSL(n,\mathbb{C})$ such that 
$$\varphi=g\cdot \varphi_n$$
almost everywhere.
\end{theorem}

The theorem is a straightforward generalization of the corresponding statement with $\mathcal{F}(\mathbb{C}^n)$ 
replaced by $\partial \mathbb{H}^3$ and $\PSL(n,\mathbb{C})$ replaced by $\PSL(2,\mathbb{C})$ 
which was proven by Thurston for the generalization of Gromov's proof of Mostow rigidity for $3$-dimensional hyperbolic manifolds. 
Our proof is a reformulation of Dunfield's detailed version \cite[pp. 654-656]{Dunfield} of Thurston's proof 
\cite[two last paragraphs of Section 6.4]{Thurston_notes} in the language of ergodic theory.

Let $T$ denote the set of $4$-tuples  in $\partial \mathbb{H}^3$ whose convex hull is a regular simplex. 
Denote by $\Lambda_{\underline{\xi}}<\mathrm{Isom}(\mathbb{H}^3)$ 
the reflection group generated by the reflections in the faces of the simplex $\underline{\xi}$. 
For $\varphi:\partial \mathbb{H}^3\rightarrow \mathcal{F}(\mathbb{C}^n)$, 
we let $T^\varphi$ be the subset of $T$ of regular simplices being mapped to maximal $4$-tuples (up to sign). 
More precisely, we set

\begin{equation*}
T^\varphi:= \left\{ (\xi_0,\dots,\xi_3)\in T \left| B_n(\varphi( \xi_0),...,\varphi( \xi_3))=\frac{1}{6}n(n^2-1)\Vol_{\mathbb{H}^3}(\xi_0,\dots,\xi_3)\right. \right\}.
\end{equation*}

Exactly the same proof as in \cite[Lemma~4.6]{BBI-Mostow} shows the following:

\begin{lem} \label{lemma: step 3} Let $\underline{\xi}=(\xi_0,...,\xi_3) \in T$. 
Suppose that $\varphi:\partial \mathbb{H}^3\rightarrow \mathcal{F}(\mathbb{C}^n)$ is a map 
such that for every $\gamma\in \Lambda_{\underline{\xi}}$, the translate $(\gamma \xi_0,\dots,\gamma\xi_3)$ belongs to $T^\varphi$. 
Then there exists a unique $g\in \PSL(n,\mathbb{C})$ 
such that $g\varphi_n(\xi)=\varphi(\xi)$ for every $\xi\in \bigcup_{i=0}^3 \Lambda_{\underline{\xi}} \xi_i$. 
\end{lem}

In the proof of Mostow Rigidity in dimension greater than or equal to $4$, Lemma~\ref{lemma: step 3} 
was sufficient to prove the corresponding Theorem~\ref{thm: step 3}. 
In dimension $3$ however, an additional difficulty is due to the fact that 
the group $\Lambda_{\underline{\xi}}$ is discrete in $\mathrm{Isom}(\mathbb{H}^3)$ 
and in particular does not act ergodically on $\mathrm{Isom}(\mathbb{H}^3)$. 
For this reason, we introduce the bigger group $\Gamma_{\underline{\xi}}$ 
which will act ergodically on $\mathrm{Isom}(\mathbb{H}^3)$ (Proposition~\ref{prop: ergodic}) 
and for which we can prove the corresponding statement of Lemma~\ref{lemma: step 3} (Proposition~\ref{prop: step 3}). 
We set
\begin{equation*}
\Gamma_{\underline{\xi}}:=\langle \Lambda_{\underline{\xi}}, \gamma_{\underline{\xi}} \rangle,
\end{equation*}
where $ \gamma_{\underline{\xi}}$ is defined as follows: 
If $\underline{\xi}=(+\infty,0,\xi_2,\xi_3)$ the isometry $\gamma_{\underline{\xi}}$ 
induces the map $\gamma_2:=z\mapsto 2z$ on $\partial \mathbb{H}^3=\mathbb{C}\cup \{ +\infty\}$. If $\underline{\xi}=(\xi_0,\xi_1,\xi_2,\xi_3)$ 
is any regular simplex, let $g\in \PSL(2,\mathbb{C})$ be an isometry such that $g\xi_0=+\infty$ and $g\xi_1=0$. 
Set then $\gamma_{\underline{\xi}}=g^{-1}\gamma_2g$. 

\begin{prop} \label{prop: step 3} Let $\underline{\xi}=(\xi_0,...,\xi_3) \in T$. 
Suppose that $\varphi:\partial \mathbb{H}^3\rightarrow \mathcal{F}(\mathbb{C}^n)$ is a map 
such that for every $\gamma\in \Gamma_{\underline{\xi}}$, 
the translate $(\gamma \xi_0,\dots,\gamma\xi_3)$ belongs to $T^\varphi$. Then there exists a unique $g\in \PSL(n,\mathbb{C})$ 
such that $g\varphi_n(\xi)=\varphi(\xi)$ for every $\xi\in \bigcup_{i=0}^n \Gamma_{\underline{\xi}} \xi_i$. 
\end{prop}

\begin{proof} For every $\underline{\xi}\in \partial \mathbb{H}^3$, 
let $S_{\underline{\xi}}$ denote the natural set of generators of $\Gamma_{\underline{\xi}}$ 
consisting of the reflections with respect to the faces of $\underline{\xi}$ and $\gamma_{\underline{\xi}}^{\pm 1}$. 
Exactly as for reflection groups, one shows that every $\gamma\in \Gamma_{\underline{\xi}}$ 
can be written as a product $\gamma=r_k \cdot \ldots \cdot r_2\cdot r_1$, where $r_i\in S_{r_{i-1}\cdot \ldots \cdot r_1 \underline{\xi}}$. 
Indeed by definition $\gamma=s_k \cdot \ldots \cdot s_2\cdot s_1$ for $s_i\in S_{\underline{\xi}}$ and we can take
\begin{equation*}
 r_i:=(s_1 s_2 \dots s_{i-1})s_i(s_1 s_2 \dots s_{i-1})^{-1}\in S_{r_{i-1}\cdot \ldots \cdot r_1 \underline{\xi}},
 \end{equation*}
where
$$
S_{r_{i-1}\cdot \dots \cdot r_1 \underline{\xi}}=(s_1 s_2 \dots s_{i-1})S_{\underline{\xi}}(s_1 s_2 \dots s_{i-1})^{-1}.
$$
Let now $\underline{\xi}$ be as in the assumption of the proposition. 
By Theorem~\ref{thm: max iff image of reg}, for every $\gamma\in \Gamma_{\underline{\xi}}$, 
there exists a unique $g_\gamma \in \PSL(n,\mathbb{C})$ 
such that $g_\gamma \varphi_n(\gamma \xi)=\varphi(\gamma \xi_i)$, for $i=0,\dots,3$. 
We need to show that $g_\gamma$ is independent of $\gamma$. 
Let $\gamma=r_k\cdot \ldots \cdot r_2\cdot r_1$ be as above. 
We prove the independence of $\gamma$ by showing $g_{r_k\cdot \ldots \cdot r_2\cdot r_1}=g_{r_{k-1}\cdot \ldots \cdot r_2\cdot r_1}$, 
where for $k=1$, the latter element of $\PSL(n,\mathbb{C})$ is $g_{\mathrm{id}}$. 
If $r_k$ is a reflection in one of the faces of the simplex $r_{k-1}\cdot \ldots \cdot r_2\cdot r_1 \underline{\xi}$ the claim follows by Lemma~\ref{lemma: step 3}. 
Up to conjugation, we can suppose that the simplex $r_{k-1}\cdot \ldots \cdot r_2\cdot r_1 \underline{\xi}$ has the form $\underline{\eta}= (+\infty,0,\eta_2,\eta_3)$. 
In the case where $r_k=\gamma_{\underline{\eta}}^{\pm 1}$, 
the simplex $r_k \underline{\eta}$ has the form $(+\infty,0,2^{\pm 1}\eta_2,2^{\pm 1}\eta_3)$ 
and in particular $\gamma_{\underline{\eta}}=\gamma_{\gamma_{\underline{\eta}}\underline{\eta}}$. 
It is thus enough to treat the case $r_k=\gamma_{\underline{\eta}}$. 
In this case, the vertices of $\gamma_{\underline{\eta}}$ are vertices of the tessellation of $\underline{\eta}$ by $\Lambda_{\underline{\eta}}$, 
which is a subgroup of $\Gamma_{\underline{\xi}}$, so the claim follows by Lemma~\ref{lemma: step 3}.
\qed 
\end{proof}

\begin{lem}\label{prop: ergodic}  For every $\underline{\xi}\in T$, the group $\Gamma_{\underline{\xi}}$ acts ergodically on $\mathrm{Isom}(\mathbb{H}^3)$.
\end{lem}

\begin{proof} We could show by hand that $\Gamma_{\underline{\xi}}\cap\PSL(2,\mathbb{C})$ is dense in $\PSL(2,\mathbb{C})$, which is equivalent to the ergodicity statement. 
We give however instead a geometric argument.  

We claim that $\Lambda_{\underline{\xi}}$ is of infinite index in $\Gamma_{\underline{\xi}}$.  
Let $\Lambda'_{\underline{\xi}}<\Lambda_{\underline{\xi}}$ be a torsion free subgroup of finite index.
Then $M:=\Lambda'_{\underline{\xi}}\backslash \mathbb{H}^3$ is a finite volume manifold and
$\xi_0,\xi_1$ are cusps of $M$.  Therefore the geodesic in $\mathbb{H}^3$ joining $\xi_0$ and $\xi_1$
maps properly to a biinfinite geodesic in $M$.
If $\Lambda_{\underline{\xi}}$ were of finite index in $\Gamma_{\underline{\xi}}$, then some finite power
of the hyperbolic element $\gamma_{\underline{\xi}}$ with fixed points $\xi_0$ and $\xi_1$ would be in 
$\Lambda'_{\underline{\xi}}$ and the geodesic joining $\xi_0$ to $\xi_1$ would map to a periodic geodesic in $M$,
which is a contradiction.

We conclude now by observing that Zariski density of $\Lambda_{\underline{\xi}}$ implies density of $\Gamma_{\underline{\xi}}$
in the usual topology; in particular $\Gamma_{\underline{\xi}}$ acts ergodically on $\mathrm{Isom}(\mathbb{H}^3)$
\qed
\end{proof}

The proof of Theorem~\ref{thm: step 3} follows now completely analogously as in \cite[Proposition~4.7]{BBI-Mostow},
taking into account Lemma~\ref{prop: ergodic}.

We have thus established that $\varphi$ is essentially equal to $g\cdot \varphi_n$. 
It remains to see that $g$ realizes the conjugation between $\rho$ and $\pi_n\circ i$. 
Indeed, replacing  $\varphi$ by $g\cdot \varphi_n$ in (\ref{equ:phiequivar}) we have 
\begin{equation*} 
(g\cdot \varphi_n) ( i(\gamma)\cdot \xi) =\rho(\gamma)\cdot( g\cdot \varphi_n) (\xi)\,,
\end{equation*}
for every $\xi \in \partial \mathbb{H}^3$ and $\gamma\in \Gamma$. The $\pi_n$-equivariance of $\varphi_n$ allows us to rewrite this equation as
\begin{equation*}
g\cdot \pi_n(i(\gamma))\varphi_n(\xi)=\rho(\gamma)\cdot( g\cdot \varphi_n) (\xi)\,,
\end{equation*}
Thus, $g\pi_n(i(\gamma))$ and $\rho(\gamma)\cdot g$ which both belong to $\PSL(n,\mathbb{C})$ act identically on the image of $\varphi_n$, 
from which we conclude that they are equal. This concludes the proof of Theorem~\ref{theorem:maxrep}.

\begin{acknowledgement} Michelle Bucher was supported by Swiss National Science Foundation 
project PP00P2-128309/1, Alessandra Iozzi was partial supported by the 
Swiss National Science Foundation projects 2000021-127016/2 and 200020-144373
and Marc Burger was partially supported by the Swiss National Science Foundation project 200020-144373.  
The authors thank the Institute Mittag-Leffler in Djursholm, Sweden, 
and the Institute for Advanced Studies in Princeton, NJ, for their warm hospitality during  the preparation of this paper.

We are very grateful to Elisha Falbel for his interest and for many useful and enlightening conversations 
during the preparation of this paper.
\end{acknowledgement}



\bibliographystyle{amsplain}

\def\cprime{$'$}
\providecommand{\bysame}{\leavevmode\hbox to3em{\hrulefill}\thinspace}
\providecommand{\MR}{\relax\ifhmode\unskip\space\fi MR }
\providecommand{\MRhref}[2]{%
  \href{http://www.ams.org/mathscinet-getitem?mr=#1}{#2}
}
\providecommand{\href}[2]{#2}

\end{document}